\documentclass[runningheads]{llncs}

\usepackage{amsfonts, amssymb, amsmath}
\usepackage{mathtools}
\usepackage{tcolorbox}
\usepackage{graphicx}
\usepackage{blkarray}
\usepackage{subfig}
\usepackage{float}
\usepackage{algorithm}
\usepackage{algorithmic}
\usepackage{enumerate}
\usepackage{cite}
\usepackage{etoolbox}
\usepackage{amsfonts}
\usepackage{stmaryrd}
\usepackage{xcolor}
\usepackage{tikz}
\usetikzlibrary{decorations.pathreplacing}
\usepackage{tkz-tab}
\usepackage{todonotes}
\usepackage{array,tikz, calc}
\usepackage{tikz-cd}
\usetikzlibrary{shapes.geometric}
\usepackage{environ}
\usepackage{qtree}
\usepackage{hyperref}
\tikzset{ampersand replacement=\&}
\usepackage{forest}
\usetikzlibrary{trees}
\tikzset{
	treenode/.style = {shape=rectangle, rounded corners,
		draw, align=center,
		top color=white, bottom color=blue!20},
	root/.style     = {treenode, font=\small, bottom color=blue!30},
	env/.style      = {treenode, font=\small},
	dummy/.style    = {circle,draw}
}
\usepackage{epsf}
\usepackage{multirow}

\definecolor{marronchene}{RGB}{217,205,180}
\definecolor{vertdragee}{RGB}{180,217,190}
\definecolor{bleulavande}{RGB}{183,180,217}
\definecolor{bleuciel}{RGB}{240,240,250}
\definecolor{rougegarance}{RGB}{217,180,180}
\definecolor{grisargent}{RGB}{210,210,210}
\definecolor{oeuf}{RGB}{225,215,210}

\definecolor{vertforet}{RGB}{50,150,50}
\definecolor{bleumarin}{RGB}{50,50,150}
\definecolor{rougerubis}{RGB}{150,50,50}
\definecolor{grisouris}{RGB}{100,100,100}
\definecolor{bleucecile}{RGB}{50,100,110}
\definecolor{vertcecile}{RGB}{50,110,80}

\newcommand{\Q}{{\mathbb Q}}
\newcommand{\F}{{\mathbb F}}
\newcommand{\Z}{{\mathbb Z}}
\newcommand{\ZZ}{{\mathbb Z}}

\newcommand{\R}{{\mathcal R}}

\newcommand{\QQ}{{\mathbb Q}}
\usepackage{ntheorem}

\newcommand{\N}{\mathcal{N}}
\newcommand{\LL}{\mathcal{L}}
\newcommand{\K}{\mathcal{K}}

\DeclareMathOperator{\Span}{span}

\usepackage{lipsum}

\usepackage{xcolor}

\title{Individual Discrete Logarithm \\ with Sublattice Reduction}
\titlerunning{Splitting Step for Composite Extension Degree Finite Fields}
\author{Haetham Al Aswad\thanks{Funded by French Ministry of Army  - AID Agence de l'Innovation de D\'{e}fense. Corresponding author: haetham.al-aswad@inria.fr} \and C\'{e}cile Pierrot}
  \institute{Universit\'{e} de Lorraine, CNRS, Inria Nancy, France.}
\date{June 2022}

\begin{document}
\maketitle
\begin{abstract}
The Number Field Sieve and 
its numerous
variants is the best algorithm to compute discrete logarithms in medium
and large characteristic finite fields. When the extension degree $n$ is composite and the characteristic~$p$
is of medium size, the Tower variant 
(TNFS) is asymptotically the
most efficient one. Our work deals with
the last main step, namely the individual logarithm step, that computes a smooth
 decomposition of a given target~$T$ in the finite field thanks to two distinct phases:
an initial splitting step, and a descent tree.

In this article,
we improve on the current state-of-the-art Guillevic's algorithm dedicated to the initial 
splitting step for composite~$n$.
While still exploiting the proper subfields of the target finite field, we modify
the lattice reduction
subroutine that creates a lift in a number field of the target $T$.
Our algorithm returns lifted elements with lower degrees and coefficients, resulting
in lower norms in the number field.
The lifted elements are not only much likely to be smooth because they have smaller norms, 
but it permits to set a smaller smoothness bound for the descent tree. 
Asymptotically, our algorithm is faster and
works for a larger area of finite fields than Guillevic's algorithm, being
now relevant even when the medium characteristic $p$ is  such that $L_{p^n}(1/3) \leq p< L_{p^n}(1/2)$. 
In practice,
we conduct experiments on $500$-bit to $2048$-bit composite finite fields:
Our method becomes more efficient as the
largest non trivial divisor of $n$ grows, being thus particularly adapted to even
extension degrees.
\end{abstract}

\section{Introduction}
\label{sec:intro}

\paragraph{Context.} Given a cyclic group $G$, a generator~$g \in G$ and a target $T \in G$,
solving the discrete logarithm problem in $G$ means finding an integer~$x$
modulo $|G|$ such that $g^x = T$. While the post-quantum competition is
ongoing, the discrete logarithm problem is still at the basis of the security of 
many currently-deployed public key protocols. This article deals with the hardness
of this problem when the considered group $G=\F_{p^n}^*$ is all the invertible elements in a finite
field. Small characteristic finite fields are no longer considered in practice
because of the advent of quasipolynomial time 
algorithms~\cite{EC:BGJT14, C:GraKleZum14, EPRINT:KleWes19}
and for this reason we focus on medium and large characteristic.
We recall the usual notation\footnote{We use $L_Q(\alpha)$ instead of $L_Q(\alpha, c)$ when the value of $c$ is not important.} $L_Q(\alpha, c) = \exp ((c+o(1)) (\log Q)^\alpha (\log \log Q)^{1- \alpha})$
 when $o(1)$ tends to~$0$ as $Q = p^n$ tends to infinity. With this
notation, we say that $p=L_Q(\alpha)$ is of medium size if $1/3 < \alpha < 2/3$ and of
large size if $2/3 < \alpha$.

\paragraph{Composite extension degrees in the wild.} 
 In the sequel, we assume that our target finite field
has a non prime extension degree~$n>1$. Let $d$ be the greatest 
proper divisor
of $n$. Considering finite fields with composite extensions is highly motivated by 
pairing-based
cryptography. Pairings first appeared in 1940 when Weil showed a
way to map points of order $r$ on a supersingular elliptic curve to an element
of order $r$ in a finite field,
 but the first algorithm to efficiently compute the Weil pairing was only published in 2004 thanks to Miller~\cite{Miller_pairing}. In the early 2000s,
efficient pairing-based protocols were presented~\cite{C:BonFra01, 
AC:BonLynSha01,DBLP:journals/joc/Joux04}
and nowadays pairings are deployed in the marketplace, for example in the
Elliptic Curve Direct Anonymous Attestation protocol that is embedded in 
the current version of the Trusted Platform Module~\cite{TPM2019}
(TPM2.0 Library), released in 2019.
The security of these protocols relies on both the discrete logarithm problem in
the group of points of a pairing-friendly elliptic curve, and on the discrete
logarithm problem in a non prime finite field, which means where the extension degree
$n>1$. Pairing constructions can work with prime extension degrees, such as $\F_{p^2}$
and $\F_{p^3}$ but composite extensions are common, such as $\F_{p^4}$,
$\F_{p^6}$ and $\F_{p^{12}}$.

\paragraph{Number Field Sieve for composite extensions.} 
The Number Field Sieve  (NFS) and 
its numerous
variants is the fastest algorithm to compute discrete logarithms in finite fields of medium
and large characteristic. It has a $L_{p^n}(1/3,c)$ complexity, where the constant $0<c<2.3$ 
depends on the variant, the characteristic and the extension degree. 
One of these variants is the Tower Number Field Sieve (TNFS), 
known to be asymptotically more efficient than NFS for some fields 
when the extension degree is composite. 
 We can couple both NFS and TNFS
 with a multiple variant -- for any finite field -- and a special variant -- for sparse characteristic finite fields only,
to obtain lower asymptotic complexities. 
The main difference between TNFS and NFS comes from
the representation of the target finite field~$\F_{p^n}$: whereas in NFS the finite field
is represented as a quotient field $\F_p[x]/(f)$ with $f$ a polynomial of degree $n$ over
$\F_p$, TNFS represents it as~$\left(\R/p\R\right)[X]/(\varphi)$ with $\R$ the ring $\R =\Z[t]/(h(t))$,
$h$ being a degree~$\kappa$ polynomial that remains irreducible modulo $p$ and $\varphi$ of degree $\eta$ such that $n = \kappa \eta$.
However, every variant of NFS, including the Tower variant, is designed around
the same steps. After the polynomial selection, that permits to construct the 
target finite field
together with (at least) two auxiliary number fields $\K_1$ and $\K_2$, the
algorithm defines a small set of ``small" elements and creates linear equations among the
discrete logarithms of these elements. This is the sieving phase. A linear algebra step
returns then these specific discrete logarithms. Finally, the individual logarithm step
that is the topic of this article concludes the algorithm. Its aim is to recover the 
discrete logarithm of an arbitrary element~$T$ in the finite field thanks to all the logarithms
already computed in the linear algebra step.
 
 Introduced in 2000 by Shirokauer~\cite{shiro},
TNFS for generic extensions was reinvestigated by Barbulescu, Gaudry and 
Kleinjung~\cite{AC:BarGauKle15}, 
proving
that the asymptotic complexity of TNFS in large characteristics is similar to NFS.
Yet in medium characteristics the complexity is even greater than $L_{p^n}(1/3)$. Kim, Barbulescu~\cite{C:KimBar16} and Jeong~\cite{PKC:KimJeo17}
proposed a method to extend TNFS to composite degree extension $n$, reaching
a $L_{p^n}(1/3)$ complexity in medium characteristics too. When $n$
has an appropriate size, this variant is faster\footnote{In medium characteristics,
NFS has a complexity of $L_{p^n}(1/3, (96/9)^{1/3})$.} than NFS, with a
 complexity of $L_{p^n}(1/3, (48/9)^{1/3})$. 
Moreover, in~\cite{PalashSarkar2019AdvancesinMathematicsofCommunications}, 
Sarkar and Singh presented a unified polynomial selection for TNFS and lowered its complexity in some cases. In~\cite{C:KimBar16,PKC:KimJeo17}, when coupled with the multiple variant or special variant
 TNFS is called MexTNFS or SexTNFS, but in this article we simply denote it by MTNFS or STNFS. 
Designing a sieving step adapted in practice for TNFS, De Micheli, Gaudry 
and Pierrot~\cite{DBLP:conf/asiacrypt/MicheliGP21} reported
in 2021
 the first implementation of TNFS and performed a record computation on a $521$-bit finite field
 with extension $n=6$. Note that computing a discrete logarithm in a finite field with extension degree
 $n>1$ is in practice harder than a discrete logarithm in a prime field of similar bitsize. 
 For instance, the last record on a prime field $\F_p$ was done with NFS in 2019 in a~795-bit finite 
 field~\cite{C:BGGHTZ20}, whereas
 the last record on a field~$\F_{p^4}$ reached a $392$-bit finite field~\cite{recordfp4}. See Table~\ref{tableau_record} for some small extension degree discrete logarithm computations.

 \begin{table}[h]
\centering
   \begin{tabular}{|c|c|c|c|}
   \hline
     $ \,$ Finite field $ \,$ &$ \,$Bitsize of $p^n \,$ &$ \,$ Year$ \,$& $ \,$Team $ \,$\\
     \hline
     $\F_p$ & 795 & 2019 & Boudot, Gaudry, Guillevic, Heninger, \\
     &&&Thom\'{e}, Zimmermann\\ 
     $\F_{p^2}$ &  595 & 2015 & $ \,$ Barbulescu, Gaudry, Guillevic, Morain $\,$ \\ 
     $\F_{p^3}$ & 593 & 2019 & Gaudry, Guillevic, Morain\\ 
     $\F_{p^4}$ & 392& 2015 & $\,$ Barbulescu, Gaudry, Guillevic, Morain $\,$ \\ 
      $\F_{p^5}$ & 324& 2017 & Gr\'{e}my, Guillevic, Morain\\ 
      $\F_{p^6}$ & 521& 2021 & De Micheli, Gaudry, Pierrot\\
      $\F_{p^{12}}$ & 203& 2013 &Hayasaka, Aoki, Kobayashi,Takagi\\
     \hline
   \end{tabular}
      \caption{\label{tableau_record} Discrete logarithm records~\cite{dldatabase}
       in finite fields for various 
      extension degrees, performed with the Number Field Sieve. TNFS is only implemented
      for the $\F_{p^6}$ record, explaining the larger field reached there.}
\end{table}
 
 \vspace{-1cm}

\paragraph{Splitting step.} All the previous results mentioned above mainly focus on
adapting new methods for the context of TNFS, to reduce the complexity of
the dominating sieving and linear algebra steps. 
However Guillevic~\cite{guillevic_descent} dealt with the individual logarithm step that remained 
at the same level of difficulty in TNFS than in NFS.
 Recall that the last step consists in two distinct phases,
first a splitting phase -- also called by some authors smoothing step -- and then a descent
tree. 
Up to this result,
the standard algorithm for initial splitting for such fields was the 
Waterloo algorithm~\cite{blakeFMV,C:BlaMulVan84}, also 
called the
extended gcd method and very similar to the
 fraction method as detailed in~\cite{C:JLSV06}.
These methods iteratively generate a pair of polynomials 
and tests 
both of them for $B$-smoothness, for a given bound $B$.
Guillevic~\cite{guillevic_descent} exploits the proper subfields of the target finite field, resulting in an algorithm
that gives much more smooth decomposition of the target in the initial splitting step.
Besides, Mukhopadhyay and Sarkar's method~\cite{DBLP:journals/ffa/MukhopadhyayS20} 
deals with at the splitting step
for small characteristic finite fields 
with composite extension degrees.~\cite{DBLP:journals/ffa/MukhopadhyayS20}
is dedicated to the Function Field Sieve and is not applicable in our context. 

\paragraph{Our work.}
In this article,
we improve on the current state-of-the-art Guillevic's algorithm dedicated to the initial 
splitting step for composite~$n$.
While still exploiting the proper subfields of the target finite field, 
and running a reduction algorithm on a well-defined lattice as in~\cite{guillevic_descent},
we manage to control the degree of the candidates for the $B$-smoothness test as in~\cite{DBLP:journals/ffa/MukhopadhyayS20}. The key idea is to use sublattices
of the original lattice of~\cite{guillevic_descent} by removing some rows and columns.
As a result, our algorithm returns number field elements with lower degrees and slightly bigger coefficients, resulting when the parameters are well set to lower norms in the number field. 
These elements are not only much likely to be smooth because they have smaller norms, 
but it allows a smaller smoothness bound for the descent tree. As a consequence 
it reduces the
height of the subsequent tree.

Besides, using the BKZ reduction algorithm instead of LLL 
allows to better fine-tune the asymptotic parameters. 
We get an algorithm that works for an asymptotic range of characteristics 
where~\cite{guillevic_descent} does not apply, namely for
$L_{p^n}(1/3) \leq p< L_{p^n}(1/2)$. In this range, 
when $p \neq L_{p^n}(1/3)$, the former asymptotic
complexity was the one of the splitting step of NFS,
whereas we get a better asymptotic complexity for 
composite extensions $n$ 
in $L_{p^n}(1/3, (3(1+\zeta -d/n))^{1/3})$, where $d$ is the largest proper 
divisor of $n$ and $\zeta$ is the value such that the infinite norm of the polynomial defining the 
number field for the lift is in~$p^\zeta$. For instance for even extension degrees,
and for the Conjugation method, we lower the asymptotic complexity
from approximately\footnote{This is the asymptotic complexity for the initial splitting step of NFS, given by Waterloo algorithm.} $L_{p^n}(1/3, 1.82)$ to approximately
$L_{p^n}(1/3, 1.14)$, which is a dramatic asymptotic improvement. Note that the extension degree $n$ is always even for finite fields of supersingular pairing-friendly curves.
Besides, we show that using BKZ instead of LLL 
allows to reach a lower complexity for the individual logarithm step when $p = L_{p^n}(1/2)$.
Moreover, we prove that in large characteristic finite fields where $p = L_Q(\alpha)$ with $\alpha \geq 2/3$, one can apply an enumeration
algorithm to the lattice instead of LLL or BKZ while keeping the same complexity 
for the individual logarithm step.
 Figure~\ref{fig:sublattices} illustrates for
even extension degrees the complexities brought by the use of LLL, BKZ, or an enumeration algorithm depending on the domain.
 Similar results are obtained for odd extension degrees.
 Table~\ref{table:interests} sums up the six existing smoothing methods
available for medium and large characteristic finite fields.

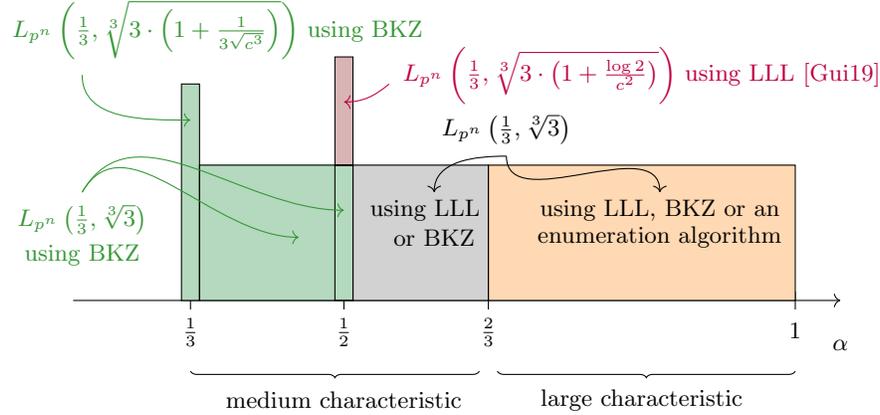
\begin{figure}[h!]
\begin{center}
\begin{tikzpicture}[scale=1.2]
    \draw[->] (2,0) -- (10.5,0);
    \foreach \x/\xtext in {3.3/{\frac{1}{3}}, 5.0/{\frac{1}{2}}, 6.6/{\frac{2}{3}}, 10/1} {
    \draw (\x,0.1cm) -- (\x,-0.1cm) node[below] {$\xtext\strut$};
    }
    \node at (10.5,-0.5) {$\alpha$};
    \fill[fill=vertdragee] (3.2,0) rectangle (3.4,2.4);
    \fill[fill=vertdragee] (3.4,0) rectangle (4.9,1.5);
     \fill[fill=vertdragee] (4.9,0) rectangle (5.1,1.5);
    \fill[fill=rougegarance] (4.9,1.5) rectangle (5.1,2.7);
    \fill[fill=grisargent] (5.1,0) rectangle (6.6,1.5);
     \fill[fill=orange!30] (6.6,0) rectangle (10,1.5);
    \draw (3.2,0) rectangle (3.4,2.4);
    \draw (3.4,0) rectangle (5.1,1.5);
     \draw (4.9,1.5) rectangle (5.1,2.7);
     \draw (4.9,0) rectangle (5.1,1.5);
     \draw (5.1,0) rectangle (6.6,1.5);
     \draw (6.6,0) rectangle (10,1.5);

      \node at (6.8,1.9) { $L_{p^n}\left(\frac{1}{3}, \sqrt[3]{3} \right)$ };
        \node at (5.9,1) {using LLL};
        \node at (6,0.7) {or BKZ};
        \draw[->,color=black] (6.8,1.6) to[out=180,in=90] (6,1.2);

 \node at (8.5,1) {using LLL, BKZ or an}; 
  \node at (8.5,0.7) {enumeration algorithm};
  \draw[->,color=black] (6.8,1.6) to[out=-90,in=90] (8.5,1.2);
 
       \draw[->,color=vertforet] (2.1,2.5) to[out=-90,in=180] (3.3,2);
    \node[color=vertforet]  at (3.6,3) { $ L_{p^n}\left(\frac{1}{3}, \sqrt[3]{3\cdot  \left(1+ \frac{1}{3 \sqrt{c^{3}}}\right)}\right)$ using BKZ};
    
	   \draw[->,color=vertforet] (2.1,1.2) to[out=60,in=180] (5,1);       
       \draw[->,color=vertforet] (2.1,1.2) to[out=60,in=180] (4.5,0.7);
    \node[color=vertforet]  at (2.1,0.9) { $ L_{p^n}\left(\frac{1}{3}, \sqrt[3]{3} \right)$};
    \node[color=vertforet]  at (2.1,0.5) {using BKZ};
        
    \draw[->,color=purple] (5.5,2.4) to[out=-140,in=0] (5,2.2);
    \node[color=purple]  at (8.3,2.5) { $ L_{p^n}\left(\frac{1}{3}, \sqrt[3]{3\cdot  \left(1 +  \frac{\log 2}{c^{2}}\right)} \right)$ using LLL \cite{guillevic_descent}};
    \draw[decorate,decoration={brace,raise=0.1cm}] (6.5,-0.7) -- (3.3,-0.7) ;
    \node at (5,-1.1) {medium characteristic};
     \draw[decorate,decoration={brace,raise=0.1cm}] (10,-0.7) -- (6.7,-0.7) ;
        \node at (8.3,-1.1) {large characteristic};
\end{tikzpicture}

\caption{\label{fig:sublattices} \emph{Complexities for the individual logarithm phase for finite fields of even extension degrees, as a function of the characteristic $p = L_{p^n}(\alpha, c)$, in JLSV$_1$.}
The height of each rectangle represents the complexity of the individual logarithm step in the corresponding range. 
The red rectangle represents the complexity 
when $\alpha =1/2$ brought by \cite{guillevic_descent}, using LLL.
Using BKZ (in green) we are able to reduce the complexity
when $\alpha =1/2$ and to reach smaller characteristics.  
When  $1/2 <\alpha < 2/3$ (in grey) the complexities are equal regardless 
of LLL or BKZ. When $\alpha \geq 2/3$ (in orange), LLL, BKZ, or an enumeration algorithm give the same complexity.
}

\end{center}
\end{figure}
\begin{table}[h!]

\begin{tabular}{ |c|m{7em}|m{8em}|m{17em}| } 
\hline
& Method & Parameters in Algorithm~\ref{our_algorithm} & Interest \\
\hline
 \multirow{2}{4em}{Previous work} & Waterloo \cite{C:BlaMulVan84} & Not supported  & The only method working for prime extensions.\\
\cline{2-4}
 & \cite{guillevic_descent} & LLL-reduction \newline with $s=0$ & Best method in practice for small extension degrees. Ex: $n =4,6,10$. Section~\ref{sec:practical}. \\
\hline
\multirow{3}{4em}{Our work} & LLL-reduction on sublattices & LLL-reduction with $s > 0$ & Best method in practice when $n$ grows. Ex: $n \geq 16$. Section~\ref{sec:practical}.   \\ 
\cline{2-4}
& BKZ-reduction & BKZ-reduction with $s=0$ & Best asymptotic algorithm for medium characteristic finite fields. Section~\ref{subsec:best_comp}.\\
\cline{2-4}
& BKZ-reduction on sublattices & BKZ-reduction with $s> 0$ & Lower norms than BKZ with $s=0$, but no change is the asymptotic complexity. Section~\ref{subsec:BKZ_method_s}. \\ 
\cline{2-4}
& Enumeration \newline algorithm & Enumeration algorithm with $s=0$ & Best asymptotic algorithm for large characteristic finite fields. Section~\ref{subsec:BKZ_method_s}\\
\hline

\end{tabular}
\caption{\label{table:interests}Smoothing algorithms for medium and large characteristic finite fields.
The crossover point when LLL on sublattices becomes better than  \cite{guillevic_descent} depends on the size of the target: examples are given for $500$-bit size finite fields.}
\end{table}

In practice, we conduct experiments on finite fields
of size $500$~bits (the current TNFS record size), 
$700$~bits (a probably reachable size using TNFS),
$1024$~bits (a key-size that can still be found in the wild) and 
$2048$~bits (a relevant key-size).  The results of all four experiments together with the code to reproduce them
 are available at the git repository~\cite{Code:smoothness}.
Our method becomes more efficient as $d$ the
largest non trivial divisor of $n$ grows, being thus particularly suitable for even
extension degrees. The results are similar for all these
sizes because the effect mainly depends on $d$ (thus $n$), and for 
this reason we only detail in this paper
our experiments on $500$ and $2048$-bit finite fields,
for extension degrees $n$ that vary from $4$ to $50$.

For instance, with a $476$-bit target finite field 
$\F_{p^{34}}$, we can lower the dimension of the $34$-dimension
lattice by removing $5$ columns and rows.
Regular lift of a target $T$ in the number field gives  elements 
with a $789$-bit norm. Applying~\cite{guillevic_descent} would 
create $540$-bit candidates to test for smoothness, whereas our
algorithm using this smaller matrix returns $517$-bit candidates 
in the number field. Bear in mind that considering FFS instead of
NFS may be relevant when dealing with
large extension degrees, depending on the whole size of the finite field.
For instance, in~\cite{Mukhopadhyay2022}, a discrete logarithm computation
is performed in a 1051-bit field with extension degree~$50$ using FFS.
However, for our $2048$-bit finite fields, NFS and its variants would likely be the most efficient available algorithms to use since the characteristic sizes are sufficiently large.

\paragraph{Outline of the article.}
In Section~\ref{sec:NFS}, we give a short refresher on the Number Field Sieve
together with the background needed on lattice reductions.
Section~\ref{sec:smoothing} presents our algorithm to compute
a candidate with a smaller norm in the number field, for the initial
splitting step. Like \cite{guillevic_descent}, our algorithm works for TNFS, MTNFS and STNFS. Then in Section~\ref{sec:asymp} we focus on
the asymptotic complexity of the splitting step, if LLL is used
for the lattice reduction. Section~\ref{sec:BKZ} deals with the impact of replacing
LLL by BKZ. In particular Corollary~\ref{lower_complexity}
gives lower asymptotic complexities for the individual logarithm
phase.
Finally 
Section~\ref{sec:practical} is dedicated to our practical results on
460 to 500-bit and 2050 to 2080-bit
finite fields with composite extensions~$n$, up to $50$.

\section{Background}
\label{sec:NFS}
From now on, $\F_{p^n}$ is the target finite field, and $n = \eta \kappa$ is its composite
extension degree. Let $d$ be the largest divisor of $n$ strictly lower than~$n$.
Since the computation of a discrete logarithm in a group can be reduced to 
its computation in one of its prime subgroups by Pohlig-Hellman's reduction, 
we work modulo $\ell$, a non trivial prime divisor of $\Phi_n(p)$, with
$\Phi_n$ the $n$-th cyclotomic polynomial.
We start with a useful definition:
\begin{definition}
Let $x$ and $B$ be two positive integers. Then $x$ is said to be B-smooth if
all its prime divisors are lower than $B$.
\end{definition}

Let us give a short refresher on the Number Field Sieve and some details 
about its Tower variant. Both NFS and TNFS follow similar steps as any index calculus algorithm.

\subsection{The (Tower) Number Field Sieve}

\paragraph{Polynomial selection.}
TNFS selects three polynomials, namely $h$, $f_1$ and $f_2$ in~$\Z[x]$.
The polynomial $h$ must be of
degree~$\eta$ and irreducible modulo the characteristic~$p$. Let $\iota$
be a root of $h$. Then we have an intermediate number
field $\mathbb{Q}(\iota)$.
The polynomials $f_1$ and $f_2$ have degree at least
$\kappa$. 
Conditions on the polynomials $h$, $f_1$ and $f_2$ permit
to define  two ring homomorphisms from $\mathcal{R}[x] = \mathbb{Z}[\iota][x]$
to the target finite field $\mathbb{F}_{p^n}$ through the number fields
$K_{1} = \QQ(\iota)[x]/(f_1(x))$ and $K_{2} =
\QQ(\iota)[x]/(f_2(x))$. This yields a commutative diagram as shown
in Figure~\ref{fig:diagram}. 

The classical NFS works with an easier polynomial selection where
we only need $f_1$ and $f_2$. The relative commutative diagram
is the same as in~Figure~\ref{fig:diagram} but with  $\mathcal{R}[x] =\Z[x]$.
Several polynomial selections for NFS are possible, giving 
each one a pair $(f_1, f_2)$ of polynomials. The most important parameters
are the size of their coefficients and their respective degrees.
In practice polynomial selections such as Conjugation~\cite{EC:BGGM15}, 
JLSV$_1$~\cite{C:JLSV06} or Sarkar-Singh's~\cite{AC:SarSin16} methods can be adapted to the TNFS setting
to obtain three polynomials $h$, $f_1$ and $f_2$ as required.

\vspace{-0.5cm}
\begin{figure}
	\centering
\[
    \begin{tikzcd}[column sep=0em]
	\& \mathcal{R} \left[ x \right] 
	\arrow{ld}[above left]{x \mapsto \alpha_1}
	\arrow{rd}[above right]{x \mapsto \alpha_2} \&
	\\
	K_{1} \supset \mathcal{R} \left[ x \right] / (f_1(x))
            \arrow{rd}[below left]{ \bmod p, \bmod \varphi}
            \& \&\
        K_{2} \supset \mathcal{R} \left[ x \right] / (f_2(x))
            \arrow{ld}{ \bmod p, \bmod \varphi}\\
            \&\mathcal{R}/\mathfrak{p}[x]/(\varphi(x)) \cong \F _{p^{n}}\&
	\end{tikzcd}
	\]
		\caption{Commutative diagram of Tower NFS. Here 
		$\alpha_{i}$ is a root 
of~$f_i$  in $K_{i}$ for $i = 1,2$ and
		$\varphi$ is of degree $\eta$ and irreducible
modulo an ideal $\mathfrak{p}$ above $p$ in~$\mathcal{R}$.
It is a common factor modulo $p$ of $f_1$ and $f_2$.
		}
		\label{fig:diagram}
\end{figure}
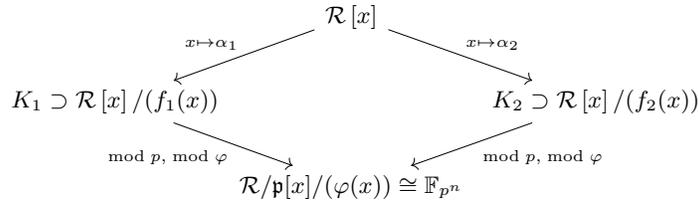
\vspace{-1cm}
\paragraph{Relation collection.}
The goal of the relation collection step is to select among a set of polynomials 
$\phi(x, \iota) \in \mathcal{R}[x]$ with a bounded degree
at the top of the diagram the candidates which produce a relation.
A relation is found if the polynomial~$\phi(x, \iota)$ mapped to $K_1$ and $K_2$ 
factors into products of prime ideals of small norms in both number fields.
The ideals of small norms that occur in these factorizations constitute the factor basis.  
To verify the $B$-smoothness on each side, one needs to evaluate the norms 
$\N_{i}(\sum_{k=0}^{\hbox{deg} f_i} a_k(\iota) (\alpha_i)^k )$ for $i = 1, 2$. 
Note that these norms are integers that can be computed thanks to resultants.
The relation collection step stops when we have enough 
relations to construct a system of linear equations that 
may be full rank. The unknowns of these equations are 
the \textit{virtual} logarithms of the ideals of the factor basis.
For the classical NFS, the relation collection is easier
and consists on the same idea, but working with univariate polynomials $\phi(x) \in \Z[x]$ instead of bivariate polynomials.

\paragraph{Sparse linear algebra.}
A good feature of the linear system created in both NFS and TNFS (there is no difference
for this step) is that the number of non-zero
coefficients per line is very small. Sparse 
linear algebra algorithms such as the block variant of Wiedemann's
algorithm~\cite{wiedemann} speeds up the computation. 
The output of the linear algebra phase is a kernel vector corresponding 
to the virtual logarithms of the ideals in the factor basis. 

\paragraph{Individual discrete logarithm.}

The final step of TNFS consists in finding the discrete logarithm of an arbitrary 
element~$T'$ in the target finite field, that we call the target element.
 This step is subdivided into two substeps: 
 an initial splitting step -- also called smoothing step -- and a descent step.
 The splitting step is an iterative process where $T'$ is first
 randomized by~$T = g^t T' \in \mathbb{F}_{p^n}^*$, where
 $g$ is a generator of~$ \mathbb{F}_{p^n}^*$ and for $t \in \Z$ chosen uniformly at random.
Values for $t$ are tested until $T$ lifted back to one 
of the number fields $K_i$ is $B_i$-smooth for a smoothness bound $B_i > B$.
We focus on this step in Section~\ref{sec:smoothing}. This step 
dominates the asymptotic complexity of the individual discrete
logarithm phase.

The second step consists in decomposing every factor of the lifted 
value of~$T$
into ideals of the factor basis for which we know the virtual logarithms. 
In our case these factors are prime ideals with norms less than a smoothness bound~$B$.
This process creates a descent tree where the root is the lift of $T$, a node is an
 ideal coming from 
the smoothing step and
the nodes below are ideals that get smaller and smaller as they go
deeper. The leaves are ultimately elements of the factor basis. 
The edges of the tree are defined as follows: 
for every node, there exists an equation between the ideal of the node and 
all the ideals of its children.

\subsection{Tools from lattice reduction algorithms}

\paragraph{SVP problem and enumeration algorithms:}
Given an Euclidean lattice~$\LL$, the SVP (Shortest Vector Problem) consists in finding the smallest non zero vector of the lattice. The best existing algorithms to solve it are exponential in the dimension of the lattice.
The family of enumeration algorithms are used in practice for small
dimension lattices. For instance the HKZ-reduction algorithm~\cite{SODA:MicWal15} finds the shortest vector of a lattice of dimension $n$ in time~$2^{O\left(n \log(n)\right)}$ and in a polynomial space complexity. 
There exists an older enumeration algorithm~\cite{STOC:MicVou10}  that is 
asymptotically faster with a time complexity in $O\left(2^{2n}\right)$,
 but the huge drawback is its exponential space complexity 
 in $O\left(2^n\right)$.

\paragraph{The LLL algorithm: }
To handle the difficulty of SVP, another problem 
was introduced, $\gamma$-SVP. This problem 
consists in looking for a short vector of the lattice, more precisely, if we denote by $\lambda_1$ the first minimum of the lattice, which is the
Euclidean norm of the shortest non zero vector of the lattice $\LL$, $\gamma$-SVP consists in 
finding $v \in \LL^*$ such that $|v| \leq \gamma \lambda_1$.

Lenstra, Lenstra, and Lovasz proposed in 1982 \cite{LLL} an algorithm that solves $\gamma$-SVP in 
polynomial time for a certain parameter $\gamma$. The algorithm takes as input a basis of $\LL$ and 
returns another basis of the lattice which has better properties, in particular the first vector of 
the returned basis is a solution to $\gamma$-SVP. Let us state some properties of the LLL algorithm:

\begin{theorem}
\label{LLL_bound}
Let $\LL$ be a lattice of dimension $n$. Let $\lambda_1$ be the first minimum of the lattice and $R$ the first vector returned by LLL when a basis of $\LL$ is given as input, then : 
\begin{itemize}
\item $\|R\|_{\infty} \leq \|R\|_2 \leq 2^{\frac{n-1}{2}} \lambda_1$.
\item $ \|R\|_{\infty} \leq \|R\|_2 \leq 2^{\frac{n - 1}{4}} \det \left( \LL \right)^{\frac{1}{n}}$.
\end{itemize}
\end{theorem}

Modifying some parameter inside LLL permits to obtain an upper bound:
 $\|R\|_2 \leq (4/3)^{\frac{n - 1}{4}} \det \left( \LL \right)^{\frac{1}{n}}$.
However setting $2$ or $4/3$  is negligible in the sequel.

\paragraph{The BKZ algorithm:}
The best approximation algorithm known in practice for large dimensions is the 
Blockwise-Korkine-Zolotarev (BKZ) algorithm, published by Schnorr and Euchner 
in 1994 \cite{DBLP:journals/mp/SchnorrE94}.
The Schnorr-Euchner's BKZ algorithm can be seen as a generalization of LLL where 
instead of considering pairs of vectors, one looks at blocks of projected vectors. 
BKZ thus 
has an additional parameter~$\beta \geq 2$
which corresponds to the considered size of block.  We denote by $\beta$-BKZ 
the algorithm 
BKZ when the integer $\beta$ is taken as parameter. 
As LLL, BKZ returns a new basis of the lattice~$\LL$ given as input,
in particular the first vector of the basis is a solution to the $\gamma$-SVP 
problem.
Roughly speaking, the higher $\beta$ is, the slower the algorithm and the 
better the output basis.
We shortly recap two theorems concerning BKZ, for further details please see~\cite{C:HanPujSte11,EC:MicWal16}.

\begin{theorem}
\label{BKZ-bound}Let $\LL$ be an Euclidean lattice of dimension $n$, and 
$R$ the first vector returned by $\beta$-BKZ applied on $\LL$, then:
\[
\|R\|_{\infty} \leq \|R\|_2 \leq 2\beta^{ \frac{n-1}{2(\beta-1)} + \frac{3}{2} } \ \det \left(\LL\right)^{\frac{1}{n}}
\]
\end{theorem}

\begin{theorem}
\label{complexity_of_BKZ}
The complexity of $\beta$-BKZ on an Euclidean lattice $\LL$ of dimension~$n$ is: 
$$
\hbox{Poly} \left(n, \hbox{size}(\LL)\right) 2^{O(\beta)}
$$
where $\hbox{size}(\LL)$ is the sum of logarithms of absolute values of the coordinates of the matrix representing $\LL$, and $\hbox{Poly}\left(n, \hbox{size}(\LL)\right)$ denotes a polynomial function in $n$ and $\hbox{size}(\LL)$.
\end{theorem}

\section{Splitting step with a smaller lattice}
\label{sec:smoothing}

In this section we study the splitting step in large and medium characteristic finite 
fields of composite extension degrees. 
For the sake of simplicity, we detail our algorithm with the classical
NFS setting, namely we consider the morphism $\mathcal{K}_i = \Q[X] / (f_i(X)) \rightarrow
\F_{p^n} \cong \F_p[X] / (\varphi(X))$. A preimage of an element $S$ in the finite field
through this morphism is called a lift, and is written $\overline S$.
We work with the classical setting because
it is easier to compare our result with~\cite{guillevic_descent} that
is written for NFS, 
but in the TNFS setting the whole algorithm works in the same way.
The goal is to improve the smoothness 
probability of the lift of $T \in \mathbb{F}_{p^n}^*$ to~$\mathcal{K}_i$ 
by constructing an adequate lattice whose reduced vectors 
define elements of $\mathcal{K}_i$ with potentially small norms, which 
are precisely the potential lifts of $T$ we are looking for.

\subsection{Splitting step with proper subfields}
\label{subsec:Guillevic}

The aim of the splitting step is to compute the
discrete logarithm of a target~$T'$ in a finite field. The
discrete logarithm is computed modulo~$\ell$, a given and precomputed integer.
The key idea of the algorithm is to replace the target $T'$ 
 by another element $T$ so that: 
\begin{enumerate}
\item $\log(T') \equiv \log(T) - t \mod \ell$ for a random known $t$.
\item the norm of the lift of $T$ in one of the number fields $K_1$ or $K_2$
 is $B$-smooth, for some predefined smoothness bound $B$ that
 is usually larger than the bound for the factor base.
\end{enumerate}
In the sequel we simply note $\K$ the number field that is chosen
to be the one where we lift the elements. We note $f$ the polynomial
defining $\K$. 

To do so, \cite{guillevic_descent} creates a lattice in $\K$ so that,
for any element $\bar S$ in the lattice its image $S$ in the finite field verifies 
the first item above, namely
$ \log(T) \equiv \log(S) + t \mod \ell$.
Performing a  lattice reduction on this set, Guillevic is then able to
produce a number field element $R$ with a small norm. 
This procedure is done over and over on $g^t T'$ 
where $t$ is chosen randomly in $[ 1, l-1 ]$ 
until the algorithm outputs an element $R$ that verifies the second item, namely 
such that its norm  is $B$-smooth.

\paragraph{Construction of a lattice thanks to proper subfields elements.}
The construction is based on the existence of elements in the finite field
for which we can deduce in advance that they have the same discrete
logarithm as the target element modulu $\ell$, a prime divisor of the multiplicative group order. Such elements are found thanks to the following lemma:
\begin{lemma}[From \cite{AC:Guillevic15}, Lemma 2]
Let $\ell$ be a divisor of $\phi_n(p)$. Let $U \in \F_{p^n}^*$ be an element that lies in a proper subfield of $\F_{p^n}$, then $\log(U) \equiv 0 \mod \ell$. 
\label{cruciallemma}
\end{lemma}

In order to construct the lattice, exhibiting an element in a proper subfield is sufficient.
Indeed, let's compute once and for all: 
\[U = g^{ \frac{p^n - 1}{p^d -1} },\] 
where  $g$ denotes a generator of the multiplicative group of $\F_{p^n}^*$ of 
order~$\ell$, and $d$ is the largest proper divisor of $n$.
Then $\{ 1, U, \dots, U^{d-1} \}$ is an $\F_p$-basis of $\F_{p^d}$. 
In particular this is done
before randomizing $T'$ as $T = g^t T'$, 
for an integer $t \in [0, \dots \ell-1 ]$.  $T$ becomes
the temporary new target.
The following elements $\{T, UT, \dots, U^{d-1}T \}$ are  $\F_p$-independent 
and every element~$R$ of the $\F_p$-vector-space spanned by 
the~$d$~previous elements verifies $\log(R) \equiv \log(T) \mod \ell$. 
Indeed, an $\F_p$-combination can be written 
as $R = \left(a_0 + a_1 U + \dots + a_{d-1} U^{d-1}\right) T$.
On the one hand $R=0$ if and only if it is the trivial combination, and on the other hand, since $a_0 + a_1 U + \dots +a_{d-1} U^{d-1} \in \F_{p^d}^*$, applying Lemma~\ref{cruciallemma} we get the desired equality.
Thus $\{T, UT, \dots, U^{d-1}T \}$ are sent to the number field~$\K$ in which they form a lattice over $\ZZ$.
Applying LLL  to it allows to find a short vector in the lattice 
that corresponds to an element $\overline{R}$ in the number field with small norm,
and such that its image~$R$  in the finite field has the same logarithm modulo~$\ell$
 as $T$.
If the norm of $\overline R$ is $B$-smooth for a predefined  bound~$B$, 
then the algorithm returns $t$ and $R$, that becomes the new target for the descent tree. Otherwise, one starts over with a new~$t$ 
 until a $B$-smooth element is found.

\subsection{Sublattices for smaller norms in the number field}
\label{subsec:methods}
The main idea presented in~\cite{guillevic_descent} for the splitting 
step in large and medium characteristic finite fields 
is to substitute the target by another one that has smaller coefficients. 
In small characteristic finite fields of composite extension 
degree, \cite{DBLP:journals/ffa/MukhopadhyayS20} replaces the target by another 
one with a smaller degree. The method we propose allows the advantages
of both worlds, supplanting the target with candidates with smaller coefficients
and smaller degrees. The key ingredient is to consider sublattices of 
the initial one.
We study the splitting step for a number field~$\K$ defined by a degree~$n$ polynomial~$f$.
The presentation is easier this way, and this matches all the polynomial selections where
at least one of the polynomial is of degree $n$. We explain in Paragraph~\ref{subsec:MTNFSandSTNFS} how to 
adapt our work  to a more general case where $\deg(f) \geq n$.

\paragraph{Description of our algorithm when $\hbox{deg}(f)=n$.}
Algorithm~\ref{our_algorithm} details our method and an implementation of Algorithm~\ref{our_algorithm} is available at \cite{Code:smoothness}. The idea consists of computing
an element~$U$ of a proper subfield to construct a lattice~$M$ of elements that
all have the same discrete logarithm as a randomized target~$T$. After a Gauss reduction on the matrix, we send its coefficients to $\ZZ$ and complete it in a square matrix~$\LL$ 
by adding elements multiple of $p$. Our algorithm
differs from~\cite{guillevic_descent} at this step: We do not apply
a reduction algorithm on the full matrix~$\LL$  but consider instead
a sublattice~$\LL'$ of~$\LL$ with a smaller dimension. 
$\LL'$ is
constructed from~$\LL$ by deleting $s$ specific 
rows and columns, with $s$~an integer
in~$\in [ 0, \ d-2 ]$ that is defined beforehand.
Applying a lattice reduction algorithm on~$\LL'$, we get a $(n-s)$-dimensional vector $(r_0, \cdots, r_{n-s-1})$ of~$\LL'$, 
and we create a candidate in $\mathcal{K}_i: \overline R = \sum_{k=0}^{n-s-1} r_k \alpha^k$,
with $\alpha$ a root of $f$.

Paragraph~\ref{subsec:setting_s} details how to set $s$. Note that
if $s=0$ then our algorithm is actually Guillevic's algorithm.
When $s>0$, the improvement  comes from
the reduction of the dimension of the vectors that are
given by the lattice reduction algorithm. Since $\LL'$ is of dimension $n-s$
instead of $n$,
the elements of the number field~$\overline R$ that are constructed are
of degree at most $n-s-1$, instead of $n-1$. Proposition~\ref{proof-of-correctness} proves the correctness of Algorithm~\ref{our_algorithm}.

\begin{algorithm}[H]
    \caption{Splitting step with sublattices for the individual logarithm in composite extension degree finite fields}
    
    \begin{algorithmic}
    \label{our_algorithm}
    \STATE \textbf{Input: } A finite field $\F_{p^n} = \F_p[X]/(\varphi)$, $n$ non prime.\\
    A lattice reduction algorithm: LLL, BKZ or an enumeration algorithm.  \\
    $\ell$ a prime divisor of $\Phi_n(p)$,
    \\
    $s \in [ 0, \ d-2 ]$.
    \\
    $\K = \Q[X]/(f)$ a number field  over $\F_{p^n}$ with $\N$ the corresponding norm.
    \\
    $v : \K \rightarrow \F_{p^n}$ a projection.
    \\
    $g$ a generator of $\F_{p^n}^*$.
    \\
    $T' \in \F_{p^n}$ the target
    \\
    $B$ a smoothness bound
    
    \STATE \textbf{Output: }$t \in [ 1, \ \ell-1 ]$, $\overline R \in \K$ such that $\log_g(v(\overline R)) \equiv t + \log_g(T') \mod \ell$, and $\N(\overline R)$ is $B$-smooth. 
    
    \begin{enumerate}
    \STATE $d \leftarrow$ the largest proper divisor of $n$.
    \STATE Compute $U = g^{ \frac{p^n - 1}{p^d - 1} }$, then $\{1, U, \dots, U^{d-1}\}$. It is an 
    $\F_p$-basis of $\F_{p^d}$.
    \STATE \textbf{Repeat:}
    \begin{enumerate}
    \STATE Choose $t \in [ 1, \ l-1 ]$ randomly.
    \STATE Compute $T = g^t T' \in \F_{p^{n}}$.
    \STATE Construct the following $d \times n$ matrix: 
    $M = 
        \begin{pmatrix}
        T \\
        U T\\\
        U^2 T \\
        \vdots \\
        U^{d-1} T
        \end{pmatrix}$ 
        
    \STATE Apply Gauss reduction to $M$ to obtain the matrix: 
    \[M_G = \begin{pmatrix}
        e_{00} & e_{01} & e_{02} & \dots & 1 &  \\
        e_{10} & e_{11} & e_{12} & \dots & * & 1 &  \\
        &&\vdots& \\
        e_{d-1 \ 0} & e_{d-1 \ 1} & e_{d-1 \ 2} & \dots & \dots & \dots & \dots & \dots& 1  
    \end{pmatrix}\]
    \STATE Send the matrix to $\Z$ and add $n-d$ rows as follows to obtain the following $n \times n$ square matrix: $\LL = \begin{pmatrix}
        p  \\
        & p \\
        &&\ddots\\
         &  &  & p  \\
        \overline{e_{00}} & \overline{e_{01}} & \overline{e_{02}} & \dots & 1  \\
        \overline{e_{10}} & \overline{e_{11}} & \overline{e_{12}} & \dots & * & 1 \\
        &&&&&&\ddots \\
        \overline{e_{d-1 \ 0}} & \overline{e_{d-1 \ 1}} & \overline{e_{d-1 \ 2}} & \dots & \dots & \dots & \dots & 1  
        \end{pmatrix}$
    \STATE  Delete the last $s$ rows and columns of $\LL$ to obtain the $(n-s) \times (n-s)$ matrix:
    $\LL' = \begin{pmatrix}
        p  \\
         &p \\
        &&\ddots\\
         &  &  & p \\
        \overline{e_{00}} & \overline{e_{01}} & \overline{e_{02}} & \dots & 1 \\
        \overline{e_{10}} & \overline{e_{11}} & \overline{e_{12}} & \dots & * & 1  \\
        &&&&&&\ddots \\
        \overline{e_{d-s-1 \ 0}} & \overline{e_{d-s-1 \ 1}} & \overline{e_{d-s-1 \ 2}} & \dots & \dots & \dots & \dots & 1  
        \end{pmatrix}$
    \STATE Apply a reduction algorithm such as LLL, BKZ, or an enumeration algorithm to $\LL'$.
    \STATE $\bar R \in \K \leftarrow$ the shortest vector returned by LLL, BKZ, or the enumeration algorithm.
    \end{enumerate}
    
    \STATE Until $\N( \bar R)$ is $B$-smooth.
    \STATE Return $t$, $\bar R$.
    
\end{enumerate}
\end{algorithmic}
\end{algorithm}

\begin{proposition}[Proof of correctness of Algorithm~\ref{our_algorithm}]
\label{proof-of-correctness}
Let $(t, \overline R)$ denote the output of Algorithm~\ref{our_algorithm} for the input $T' \in \F_{p^n}$. Define $R$ as the projection of~$\overline R$ in $\F_{p^n}$, and $T = g^t T'$ as in Algorithm~\ref{our_algorithm}. Then $\log_g(R) \equiv \log_g(T) \mod \ell$.
\end{proposition}

\begin{proof}
We adopt the same notations as in Algorithm~\ref{our_algorithm}. We represent $\K = \Q[X]/(f)$ as $\Q(\alpha)$. 
Thus for any integer $0<\mu \leq \deg(f)$, a $\mu$-dimensional vector representation 
of an element in $\K$ is relative to the independent family $\{1, \alpha, \alpha^2, \dots, \alpha^{\mu-1}\}$.
It is sufficient to prove that $\bar R$ is represented by a vector that is in the lattice spanned by 
the rows of $\LL$: $\Span\{\LL\}$. 
Indeed, as stated in \cite{guillevic_descent}, any integer linear combination of the rows of $\LL$ 
represents an element that once mapped to $\F_{p^n}$ is represented as an $\F_p$-combination
of the rows of $M$. It can be written as $uT$ where $u \in \F_{p^d}^*$, and thus its logarithm
is equal to $\log(T) \mod \ell$ by Lemma~\ref{cruciallemma}.

It remains to prove that $\bar R$ admits a vectorial representation that is an element of the lattice $\Span\{\LL\}$.
Let $0 \leq s \leq d-2$ as in the input of Algorithm~\ref{our_algorithm}. 
Any vector in $\Span\{\LL'\}$ is a representation of an element in $\K$ in the independent 
family $\{1, \alpha, \dots, \alpha^{n-s-1}\}$. Let 
$\bar R := r_0 + r_1 \alpha + \dots + r_{n-s-1} \alpha^{n-s-1}$ and
 $v' = (r_0, r_1, \dots, r_{n-s-1}) \in \Span\{\LL'\}$ its vectorial representation. 
 Since $v'$ is an integer linear combination of the rows of $\LL'$, we can write 
 $v' = \sum_{i=1}^{n-s} a_i L_i'$ where $L_i'$ denotes the $i$-th row of $\LL'$ and $a_i$ an integer for $0 \leq i \leq n-s$.

Similarly denote by $L_i$ the $i$-th row of $\LL$. For each $0 \leq i \leq n-s$, $L_i$ is the 
concatenation of $L_i'$ followed by $s$ zeros. Indeed, $\mathcal{L}$ is a lower triangular matrix, hence, after deleting the $s$ last rows of $\LL$, 
the $s$ last columns are all zeros. 
Thus the $n$ dimensional vector $v := \sum_{i=1}^{n-s} a_i L_i$ is equal to 
$v = (r_0, r_1, \dots, r_{n-s-1}, 0, \dots 0)$, and $v$ is in the lattice $\Span\{\LL\}$.
This concludes the proof since $v$ is a vector representation of the element 
$\bar R$: $\bar R = r_0 + r_1 \alpha + \dots r_{n-s-1} \alpha^{n-s-1} + 0 \alpha^{n-s+1} + \dots + 0 \alpha^{n-1}$.
\end{proof}  

\paragraph{Euclidean norms versus norms in the number field.}
Looking at sublattices
of a given lattice to find shorter norms might seem counterintuitive:
indeed, since smaller coefficients for a given vector~$v$ imply a smaller norm in
the number field for the related element constructed with $v$, our aim is to find a short vector
of~$\LL$. Considering a
sublattice~$\LL'$ may thus result in missing very short vectors that
 live in $\LL \setminus \LL'$. Indeed we run the risk of loosing the 
smallest vectors of the lattice and thus outputting an element with a greater 
Euclidean norm. However, the subtlety lies in the difference between the
 Euclidean norm (or the infinity norm) and the norm~$\N$ 
 defined over the number field~$\K$:
whereas $\N$  is sensible 
to the coefficients size \emph{and} to the degree of the polynomial, 
the Euclidean norm and the infinity norm are sensible to the coefficients sizes only.

For instance,  if $P_1 = 1 + \alpha + 3 \alpha^2$ and $P_2 = 1 + \alpha + 3 \alpha^{50}$ 
are elements of the number field $\K$, then $P_1$ and $P_2$ have both the same Euclidean and infinity norms, but $\N(P_2)$ should be much greater than $\N(P_1)$.
In practice, for all the experiments we run in Section~\ref{sec:practical},
we see that our sublattices don't give shorter vectors than the original full dimension
lattice~$\LL$. However the elements in the number fields that are constructed from
the output vectors benefit from the large number of zero coefficients at the end,
meaning a decrease in the degree, that leads to lower the norms when $n$ is large, 
as we observe.

Thus by considering a sublattice we try to balance two quantities:
we accept slightly greater coefficients and ask in return for a smaller degree. 
As a result, our algorithm returns lifted elements $\overline{R}$ with lower norms in the number field, 
as we show both  asymptotically~in~Section~\ref{sec:asymp}
and in practice~in~Section~\ref{sec:practical}. We give in Appendix~\ref{Appendix:Example} a concrete application of Algorithm~\ref{our_algorithm} on a finite field of extension $28$ and another application on a finite field of extension $50$.

\subsection{Dimension of the sublattice}
\label{subsec:setting_s}
As seen above, the dimension of the sublattice plays a key role 
in the norm of the output candidates in the number field~$\K$.
This dimension, which is $n-s$ is monitored by a parameter~$s$,
equal to the number of rows and columns we erase from the original
lattice~$\LL$. For this reason, $s$ is clearly an integer greater or
equal to~$0$. Besides, we cannot take $s$ larger than $d-2$.
Indeed, if we delete the last $s=d-1$ rows and columns from the matrix,
that would leave us, once the lattice is mapped to the finite field,
with a sub-vector-space of dimension $1$.
This would generate a trivial algorithm where we would get at the end
either a trivial element~$0$ in the finite field or
the element given by the last row of the matrix, multiplied
by a constant factor. The precise analysis that permits to balance
the risks and benefits of lowering the dimension of the lattice and
correctly tune $s$ is given in Section~\ref{sec:asymp}.
It leads to the following theorem that tells how to choose~$s$ before running the algorithm.

\begin{theorem}
\label{theo:setting_s}
Let $p$ be the characteristic of the finite field, $n$ its extension degree, $d$ the
largest divisor of $n$ and $\zeta \in [0,1]$ 
the parameter such that $\|f\|_{\infty} = p^{\zeta}$ where $f$ is the polynomial defining the number field.
Let $s_1 = n - \sqrt{  \frac{2 (n-d) n \log p }{n \log 2  + 2 \zeta \log p}} $.
The best asymptotic complexity is reached for Algorithm~\ref{our_algorithm}
with LLL when $s$ is defined as follows:
\begin{itemize}
\item If $s_1 <0$, then $s=0$.
\item If $0 \leq s_1 \leq d-2$ then $s = 
\lfloor s_1 \rfloor$ or $s = \lceil s_1 \rceil$.
\item If $s_1 > d-2$, then $s = d-2$.
\end{itemize}
\end{theorem}

\subsection{Variant for MTNFS and STNFS}
\label{subsec:MTNFSandSTNFS}

The algorithm described above works for NFS and TNFS for composite extension
degrees. It is a natural question to wonder whether it applies 
to TNFS when coupled with a multiple variant or a special variant. To avoid 
burdening this article with long details, we simply
give guidelines concerning our way to answer this question, without any long explanation on both MTNFS and STNFS.

Using a multiple variant~\cite{barbulescu:hal-00952610,EC:Pierrot15,C:KimBar16,PKC:KimJeo17}  does not affect our result here, and one can apply it almost directly, as there is no particular
way to deal with a multiple diagram during the initial splitting step. The number of potential number fields to lift in increases, but the idea remains the same:   lift your target element from the target
finite field to this number field with the lower norms.

When using a special variant~\cite{PAIRING:JouPie13,C:KimBar16,PKC:KimJeo17} for a sparse characteristic, the number field with the lower norms is the one given by the
polynomial with small coefficients but degree $\lambda n$, where $\lambda$
is a constant that depends on the target (pairing) finite field.
We need then to slightly modify the lattice in our algorithm to deal with this larger
degree. The following paragraph tackles this issue.

\paragraph{Construction of the lattice when  $\hbox{deg}(f)>n$.}
When the polynomial selection gives two polynomials $f_1$ and $f_2$
with different degrees and sizes for the coefficients, the general idea
for the individual logarithm step is to choose to lift the target in the number field
that naturally shows the smaller norms. The two polynomials are at least
of degree~$n$ since they share a common factor $\varphi$ of degree~$n$ defining
the target finite field, but one of it can have a strictly higher degree, for instance~$2n$, or even greater with the special variant.
Our algorithm applies to this more general context when the extension degree 
of the number field~$\K$ in which we lift the target elements is greater than the extension degree of the finite field. 
Indeed, let us assume that $\K =  \Q[x]/(f)$ and $\deg(f)=\tilde n \geq n$, 
then we look for $s$ over $[ 0, \tilde n-n+d-2]$ and we construct the following lattice of dimension $\tilde n \times \tilde n$ instead of the lattice $\LL$ in Algorithm~\ref{our_algorithm}:
\[\widetilde{\LL} = \begin{pmatrix}
       p & 0 & 0 & \dots & \dots & \dots & \dots & \dots &  &  & \\
          \vdots & \ddots \\
        0 &  \dots & p & 0 & \dots & \dots  &  &  & & \\
        \overline{e_{00}} & \overline{e_{01}} & \dots & 1 & 0 & \dots & \dots &  & & \\
        \overline{e_{10}} & \overline{e_{11}} & \dots & * & 1 & 0 & \dots & \dots & & \\
                          &                   &                   &        &   & \ddots &  & & & \\
        \overline{e_{d-1 \ 0}} & \overline{e_{d-1 \ 1}} &  & & \dots && 1& &  & \\  
        \varphi(x) & &  & & && &1 &  & \\
         & \ddots & & & & & & & &   \ddots & \\
         & & x^{\tilde n -n-1} \varphi(x) & & & && & & & 1 \\
	\end{pmatrix}.\]

\section{Asymptotic analysis with LLL as lattice reduction algorithm}
\label{sec:asymp}

In this section our aim is to determine the asymptotic optimal choice for~$s$: 
For a given lattice~$\LL$, we want to set~$s$ so that the algorithm outputs the
element with the smallest possible norm in the number field.
In other words we seek for the optimal sublattice. Note that in 
this section we assume that LLL is the lattice reduction algorithm that
is run. In particular in Paragraph~\ref{LLL_not_working},
we underline that neither our method nor Guillevic's one 
has any interest with LLL when the characteristic is in the lower part
of the medium characteristic area, namely when $p=L_Q(\alpha)$ 
with~$1/3 < \alpha < 1/2$.
In Paragraph~\ref{subsec:sublattice} we determine the
optimal choice for~$s$ and propose a criteria on the polynomial
selections that are concerned by our improvement.
Working with BKZ gives better
results as there is lighter restriction on~$\alpha$. For this reason we
only give the full asymptotic complexity for BKZ in~Section~\ref{sec:BKZ}, not LLL.

\subsection{Norms in the number field of the output of LLL}
\label{LLL_not_working}
Let $s \in [ 0, \ d-2 ]$ be an integer. 
We denote by $\overline{R_1}\in \K$ the candidate created thanks to the first
vector of the output of LLL in one loop of Guillevic's algorithm,
 and by $\overline{R_2}$ its counterpart in our algorithm.
 To study the quantities $\N \left(\overline{R_1}\right)$ and $\N\left(\overline{R_2}\right)$,
we start by recalling a useful bound on norms in a
number field.
For any $\overline{R} \in \K$:
\begin{equation}
\label{bound-in-number-field}
\N(\overline{R}) \leq  \left( \deg(\overline{R})+1 \right)^{\frac{\deg(f)}{2}} \left( \deg(f)+1 \right)^{\frac{\deg(\overline{R})}{2}} \|\overline{R}\|._{\infty}^{\deg(f)} \|f\|_{\infty}^{\deg(\overline{R})}
\end{equation}
We recall as well the following formula 
where as usual $Q = p^n$:
\begin{equation}
\label{formula_n}
n = \frac{1}{c} \left( \frac{\log Q}{\log \log Q}\right)^{1-\alpha}
\end{equation}

In the sequel, we assume that \emph{$\deg(f) = n$} and $\|f\|_{\infty} = p^{\zeta}$
for some $\zeta \in [0, 1]$, where $f$ is the polynomial defining the number field.
Note that $\zeta$ depends on the polynomial selection, 
and typical value are for instance $0$, $1/2$ or $1$.
Applying Theorem~\ref{LLL_bound} and 
Equation~\eqref{bound-in-number-field} while 
keeping in mind that $\deg(\overline{R_1}) = n-1$ and $\deg(\overline{R_2}) = n-s-1$, we 
deduce the following bounds on $\overline{R_1}$ and $\overline{R_2}$:
\[
\begin{array}{rcl}
 \N \left( \overline{R_1} \right) &\leq& 
 n^{\frac{n}{2}} (n+1)^{\frac{n-1}{2}} 2^{n \frac{n-1}{4}} p^{(1 + \zeta)n - d - \zeta }\\
\hbox{and} \quad  \N \left( \overline{R_2} \right) &\leq& (n-s)^{\frac{n}{2}} (n+1)^{\frac{n-s-1}{2}} 2^{n \frac{n-s-1}{4}} p^{n \frac{n-d}{n-s} + \zeta (n-s-1)}.
\end{array}
\]

Our aim is to minimize the second bound in the variable $s$.
We start by proving that the combinatorial factors in the bounds,
namely
$ n^{\frac{n}{2}} (n+1)^{\frac{n-1}{2}} $ and  $(n-s)^{\frac{n}{2}} (n+1)^{\frac{n-s-1}{2}}$
 are negligible with respect to the other factors, as soon as~$\alpha > 0$. 
Indeed, on the one hand thanks to Equality~\eqref{formula_n},
$ \log \left( n^n \right)$ is upper-bounded by
$(- \log c) / c \left(\log Q / \log  \log Q  \right)^{1 - \alpha} 
+ (1-\alpha)/c (\log Q)^{1 - \alpha} \left( \log\log Q  \right)^{\alpha} $, thus $n^n$ is in $L_Q(1-\alpha)$. On the other hand $p^{(1 + \zeta)n - d - \zeta }$ and $p^{n \frac{n-d}{n-s} + \zeta (n-s-1)}$ are lower bounded by $p^{\frac{n}{2} - 1} = L_Q(1)$.
Moreover, it is easy to see that the factor in $2^{n^2}$ that appears in both bounds is in $L_Q(2 (1-\alpha))$. 
It means that whenever $\alpha > 1/2$ this factor is negligible compared to the one in $L_Q(1)$, whenever $\alpha = 1/2$ it is in $L_Q(1)$, and whenever $\alpha < 1/2$ it dominates over $L_Q(1)$. We sum up this paragraph as follow.
Let $\overline{R}$ be the output of Algorithm~\ref{our_algorithm}: 
\begin{itemize}
\item if $\alpha > 1/2$, then $\N \left( \overline{R} \right) = O \left(p^{n \frac{n-d}{n-s} + \zeta (n-s-1)}  \right)$.
\item if $\alpha = 1/2$, then $\N \left( \overline{R} \right) = O \left(2^{n \frac{n-s-1}{4}} p^{n \frac{n-d}{n-s} + \zeta (n-s-1)}  \right)$.
\item if $\alpha < 1/2$, then the extension degree becomes too large and both bounds, Guillevic's ($s=0$) and ours become dominante with respect to $L_Q(1)$. Hence our method  -- including Guillevic's one -- is not better than a regular and simple lift of the target, without any lattice reduction. 
Indeed Inequality~\eqref{bound-in-number-field} directly states that 
a norm of any element coming up from the finite field is bounded by $L_Q(1)$. 
As far as we know,
this limitation of LLL was not made explicit in the literature.
\end{itemize}

The asymptotic complexity obtained with LLL is given in Theorem~\ref{complexity_friabilisation}.
The curious reader can find the whole complexity analysis for $s=0$ in~\cite{guillevic_descent}.

\subsection{Optimal sublattice dimension}
\label{subsec:sublattice}
We consider the bound on the norm of the output of Algorithm~\ref{our_algorithm}, namely $2^{n \frac{n-s-1}{4}} p^{n \frac{n-d}{n-s} + \zeta (n-s-1)}$ where we only neglect the combinatorial factors. We minimize this bound in $s$, thus proving Theorem~\ref{theo:setting_s} given in Section~\ref{sec:smoothing}

\begin{proof}
We introduce the following function in $s$: 
\[
h : s \mapsto 2^{n \frac{n-s-1}{4}} p^{n \frac{n-d}{(n-s)} + \zeta (n-s-1)}.
\]
We look for an integer $s_{opt} \in [0, \ d-2 ]$ such that $h(s_{opt}) = \min_{s \in [ 0, \ d-2 ] } \{ h(s) \}$. 
A computation done with SageMath 
gives the following table of variation:

\begin{center}
\begin{tikzpicture}[scale=0.5]
   \tkzTabInit{$s$ / 1 , $h$ / 1.5}{,$s_1$, $n$, $s_2$,}
   \tkzTabVar{+/, -/, R/, +/, -/}
\end{tikzpicture}
\end{center}

where $s_1 = n - \sqrt{  \frac{2 (n-d) n \log p}{n \log 2 + 2 \zeta \log p}  }$ and $s_2 > n$. Thus $h$ decreases then increases on $[0, n]$ and reaches its minimum in $s_1$.
\end{proof}

The above result explicits the optimal sublattice to construct when we use LLL as
a lattice reduction algorithm: 
\begin{itemize}
\item Either the optimal lattice is already the (full) one of dimension $n$ given in~\cite{guillevic_descent}.
\item Or the optimal lattice is given by a formula, stating how many vectors we should erase.
\item Or the optimal one is when we withdraw as many vectors as we can, which means
$d-2$.
\end{itemize}
With a given polynomial selection, and thus a fixed parameter~$\zeta$,
a natural question is whether we need to choose a sublattice or the full
lattice. To answer this question we give a simple condition on $\zeta$
that ensures that the optimal sublattice  is a strict sublattice. 
First we remark that in Theorem~\ref{theo:setting_s} if $s_1 \geq 1$ then $s_{opt} > 0$. 
So we study the condition $s_1 \geq 1$.
\[
\begin{array}{ccc}
s_1  \geq 1& \iff& n - \sqrt{  \frac{2 (n-d) n \log p }{n \log 2  + 2 \zeta \log p }  } \geq 1
\\
&\iff &(n-1)^2 \geq \frac{2 n (n-d)}{n \cdot \frac{ \log 2}{\log p } + 2\zeta} 
\\
&\iff& n \cdot \frac{ \log 2}{\log p } + 2 \cdot \zeta \geq \frac{2 n(n-d)}{(n-1)^2}
\\
&\iff& \zeta \geq \frac{n(n-d)}{(n-1)^2} - \frac{n \log 2 }{2\log p }.
\end{array}
\]
Thus for polynomial selection methods that outputs such a $\zeta$, our algorithm with LLL offers lower norms than~\cite{guillevic_descent} with LLL. For instance if we deal with even extensions, then $d=n/2$ and our algorithm is asymptotically better whenever
$\zeta \geq \frac{1}{2} \left(\frac{n}{n -1}\right)^2 - \frac{n \log 2}{2\log p}.$ It is
sufficient to have:
\begin{equation}
\label{eq:zeta}
\zeta \geq \frac{1}{2} \left(\frac{1}{1 -1/n}\right)^2.
\end{equation}

\begin{example}
JLSV$_1$ polynomial selection presented in~\cite{C:JLSV06} is a theoretical 
corner case for our method:
it outputs two polynomials $f_1$ and~$f_2$
with both degree~$n$ and coefficients such that  $\|f\|_{\infty} = \sqrt{p}$, 
namely $\zeta = 1/2$, which is the limit obtained in \eqref{eq:zeta}
when $n$ tends to infinity. Note that JLSV$_1$ is useful in the TNFS setting both 
in theory and in practice. The question whether in practice our method lowers the norms
for this polynomial selection
for current relevant sizes of finite fields is the topic of Section~\ref{sec:practical}.
\end{example}

\section{Asymptotic analysis with BKZ as lattice reduction algorithm}
\label{sec:BKZ}

This section
details the asymptotic analysis of our
algorithm when $s=0$ 
and when we use BKZ instead of LLL.
Indeed, recall that with LLL, this algorithm is asymptotically meaningful in finite fields where $\alpha \geq 1/2$. The idea is to overcome
this difficulty, that comes from 
$2^{n \frac{n-1}{4}} = L_Q \left( 2(1-\alpha) \right)$ in the bound of the norms,
by looking at an algorithm providing another term for this bound.
We show in this section that BKZ permits to extend the range of
application of the algorithm. Besides
it leads to a better asymptotic complexity for the initial splitting step.

\subsection{Fine tuning the parameter $\beta$ in BKZ when $s=0$}
\label{choice-of-beta}
Let $\beta$ be an integer in $ [ 2, n ]$ that denotes the block size in BKZ, $s=0$ and write again $\deg(f) = n$, and $\|f\|_{\infty} = p^{\zeta}$ where $f$ is the polynomial defining the number field. Let $\bar R$ be the element in the number field
constructed thanks to the coefficients of the first vector of the basis output 
by BKZ in Algorithm~\ref{our_algorithm}. Thanks to Theorem~\ref{BKZ-bound} 
and to the usual bound of a norm in a number field given by the resultant:
\begin{equation}
\label{norm_with_BKZ}\mathcal{N} 
\left(\bar R\right) \leq 2^n \beta^{n \left( \frac{n-1}{2(\beta-1)} + \frac{3}{2} \right)}  p^{n-d + \zeta (n-1)}.
\end{equation}
The combinatorial factors are negligible in the considered characteristic range.
We choose the largest $\beta$ under the constraint that $\beta$-BKZ 
stays asymptotically negligible compared to $L_Q(1/3)$.
Indeed, such a $\beta$ would neither increase the complexity of the initial
splitting step -- that is in $L_Q(1/3)$
nor the individual logarithm phase. 

From Theorem~\ref{complexity_of_BKZ} we look for the largest $\beta$ such that 
$ \hbox{Poly} \left(n, \hbox{size}(\LL)\right) 2^{O(\beta)}$ is negligible with respect to $L_Q(1/3)$, where $\hbox{size}(\LL)$ denotes the sum of logarithms of absolute values of the coefficients of our input matrix $\LL$.
On the one hand $\LL$ has coefficients all bounded by $p$, thus $\hbox{size}(\LL) \leq n^2 \log p$ and $ n^2 \log p = O(\log Q)$ from which we deduce that $\hbox{Poly} \left(n, \hbox{size}(\LL)\right)$ is negligible with respect to $L_Q(1/3)$.
On the other hand writing $\beta = n^x$ and using Equality~\eqref{formula_n} 
we get $\log (2^\beta) = \log 2/c^x (\log Q/\log\log Q)^{x(1-\alpha)}$. 
We deduce that $2^\beta$ is negligible compared to~$L_Q(x(1-\alpha))$, and likewise $2^{O(\beta)}$ is negligible compared to~$L_Q\left(x(1-\alpha)\right)$. 
We set $x$ the largest possible number such that $x(1-\alpha)\leq 1/3$ keeping in mind that $x$ must be smaller than $1$ since $\beta = n^x$ must be smaller than n. This gives the following choice:
$$
x = \left|
\begin{array}{ccc}
1 &\text{if}& \alpha \geq 2/3\\
\frac{1}{3 (1-\alpha)} &\text{if}& \alpha < 2/3
\end{array}
\right.
$$

\paragraph{Summary for the choice of the parameter~$\beta$:}
\begin{itemize}
\item When $\alpha > 2/3$, we are dealing with finite fields with large 
characteristics relatively to the size of $n$, so the extension degree, which is small,
gives a lattice~$\LL$ of small enough dimension so that we
can directly run an enumeration algorithm on it to find the shortest vector.
Indeed, setting $\beta = n$ in BKZ means calling an oracle to solve 
 SVP on the whole lattice, which are in practice enumeration 
algorithms such as
 Kannan-Fincke-Pohst algorithms~\cite{Fincke1985ImprovedMF,kannan87} or 
more recent
techniques as developed in~\cite{SODA:MicWal15}.
The complexity of 
\cite{SODA:MicWal15} is $2^{O\left(n \log(n)\right)} = L_Q(1-\alpha)$. 
This complexity is negligible with 
respect to the complexity of the individual logarithm step which is in $L_Q(1/3)$ 
as we see in the sequel.
\item When $\alpha = 2/3$, setting $\beta=n$ in BKZ is the good option too. However~\cite{SODA:MicWal15} becomes non negligible 
but~\cite{STOC:MicVou10}\footnote{The counter part of this enumeration algorithm is its exponential space complexity.}  that has a time complexity in $O\left(2^{2n}\right)$ stays negligible with respect to $L_Q(1/3)$.
\item When $\alpha < 2/3$, the extension degree and thus the dimension of the lattice becomes larger, and looking at blocks in BKZ becomes mandatory.
We propose to set $\beta = n^{(3(1 - \alpha))^{-1}}$ (which is strictly lower than $n$). The complexity of $\beta$-BKZ remains negligible compared to $L_Q(1/3)$.
\end{itemize}

\subsection{Norms in the number field of the output of BKZ}
\label{BKZ_norm_output}
Now that $\beta$ is set, we evaluate the norm of the element~$\bar R$ in the number field 
that corresponds to the first
vector of the matrix output by BKZ.

We start with large characteristic finite fields. When $\alpha \geq 2/3$ 
the idea is to apply an enumeration algorithm outputting elements of 
norms $ n^{n/2} p^{n -d + \zeta(n-1)} = L_Q\left(1, 1 + \zeta - d/n\right)$. 
Indeed, Minkowski's theorem brings $ ||R||_{\infty} \leq n^{1/2} p^{(n-d)/n}$
where $R$ is 
 the shortest vector.

Let us focus at Equation~\eqref{norm_with_BKZ} that gives a bound on the norms in the medium characteristic case. When $\alpha < 2/3$,
we set $\beta = n^x$ with $x = \frac{1}{3 (1-\alpha)}$. 
First, as above $p^{n-d + \zeta (n-1)}~=~L_Q\left(1, 1 + \zeta - d/n\right)$ and
 $2^n \leq L_Q(1-\alpha)$ is negligible compared to $L_Q(1)$ whenever $\alpha>0$. Second let us have a look at~$D = \beta^{n \left(\frac{n-1}{2(\beta-1)} + \frac{3}{2}\right) }$, and study its size.
$D = n^{ \frac{xn}{2} \left(\frac{n-1}{n^x - 1} + 3\right) }$,
by the mean value theorem applied to the 
function $f:y \mapsto y^x$, where $x<1$ on the interval $[1, n]$ we have: 
$ (n-1)\cdot(n^x - 1)^{-1}\leq (n^{1-x})/x$, 
which yields $D\leq n^{ \frac{n^{2-x}}{2} + \frac{3nx}{2}}.$
We evaluate this last quantity in two steps:

\begin{itemize}
\item $ n^{\frac{n^{2-x}}{2}} 
= L_Q \left( (2-x)\cdot (1-\alpha), (1-\alpha)\cdot (2 c^{2-x})^{-1} \right)$ where

 $(2-x)\cdot (1-\alpha) 
 = 2(1-\alpha) - 1/3$. 
As $2(1-\alpha) - 1/3\leq 1$ 

$\iff \alpha \geq 1/3$, we identify three cases:
	\begin{itemize}
	\item If $\alpha >1/3$, then $n^{\frac{n^{2-x}}{2}}$ is negligible compared to $L_Q(1)$.
	\item If $\alpha =1/3$, then $n^{\frac{n^{2-x}}{2}} = L_Q \left(1, (1-\alpha) \cdot (2 c^{2-x})^{-1}\right) $.
	\item If $\alpha < 1/3$, then the bound is no longer asymptotically significant because a simple lift in the number field of any element of the finite field has norm of size at most $L_Q(1)$.
	\end{itemize}
	
\item $n^{3nx/2}$ is negligible compared to the first factor $n^{\frac{n^{2-x}}{2}}$.
\end{itemize}
Let us compare LLL and BKZ and summarize our result up to now.

\paragraph{Six areas for the characteristics.} Here are the different areas
and the summary of the behavior of LLL and BKZ on the norms of the output
elements, depending on the size
of the characteristic, from the smallest ones, to the largest ones. 
\begin{itemize}
\item If $\alpha <1/3$ then neither LLL nor BKZ gives lower
norms than an easy lift from the finite field to the number field. 
\item If $\alpha = 1/3$, then LLL is not relevant but BKZ outputs
elements in~$\K$ with a norm bounded by $L_Q (1,1 + \zeta - d/n + (1-\alpha) \cdot (2 c^{2-x})^{-1})$.
\item If $ 1/3 < \alpha < 1/2$, then LLL is not relevant but BKZ
 provides a bound which is $L_Q (1,1 + \zeta - d/n)$.
 \item If $\alpha = 1/2$, the bound for the norm of the number field element element given
 by LLL is
 $L_Q \left(1, 1 + \zeta - d/n +  c^{-2} \log 2 \right)$ 
 while with BKZ we can get a lower bound $L_Q \left(1,1 + \zeta - d/n \right)$.
\item If $1/2 < \alpha < 1$, then the two bounds given by LLL and BKZ are equivalent and are in $L_Q(1,1 + \zeta - d/n)$.

\item If $2/3 \leq \alpha < 1$, an enumeration algorithm can replace LLL or BKZ and outputs norms in $L_Q(1,1 + \zeta - d/n)$ as well.
\end{itemize} 

\subsection{New asymptotic complexity for the individual logarithm phase}
\label{subsec:best_comp}
Since the initial splitting step dominates in terms of complexity the descent phase,
 the asymptotic complexity of the individual logarithm step is  the complexity of the 
step we are studying. As seen in~Paragraph~\ref{BKZ_norm_output},
some characteristic ranges and polynomial selections permit
to lower the norms, and thus to lower the individual
logarithm phase complexity.
\\
Recall that our choice of $\beta$ ensures that $\beta$-BKZ is of negligible complexity
compared to the complexity of the ECM smoothness test done to see if 
the norm is $B$-smooth or not, in each loop. To conclude on the total 
asymptotic complexity of this step, we must estimate the number of loops 
required to find a $B$-smooth element. To do so we recall two useful theorems 
concerning the probability of smoothness and the running time to find a smooth element.

\begin{theorem}[Canfield, Erdos, Pomerance]\cite{canfield1983problem}
\label{canfield}
\\
Let $(\alpha_1, \alpha_2, \ c_1, \ c_2) \in [0, \ 1]^2 \times [0, \ +\infty[^2$ 
such that $\alpha_1 > \alpha_2$ or ($\alpha_1 = \alpha_2$ and $c_1>c_2$). Denote by $\mathbf{P}$ the probability that a natural \textbf{random} number smaller than $A = L_Q(\alpha_1, \ c_1)$ to be $B = L_Q(\alpha_2, c_2)$-smooth. Then: 
$$
\mathbf{P}^{-1} = L_Q\left(\alpha_1 - \alpha_2, \ \left(\alpha_1 - \alpha_2\right) \frac{c_1}{c_2}\right).
$$
\end{theorem}

Let the smoothness bound be written as $B = L_Q(\alpha_B, \ c_B)$.
Theorem~\ref{complexity_friabilisation} that mostly
comes from~\cite{guillevic_descent} states the best choice on $B$.
To find this $B$ the key idea is to balance two different effects when $B$ increases.
 On the one hand,
the probability of an element $\bar R$ to be $B$-smooth increases.
On the other hand, the $B$-smoothness test by ECM becomes more costly. 

\begin{theorem}
\label{complexity_friabilisation}
Let $\bar R$ be an element of the number field~$\K$ constructed thanks to the output of LLL or $\beta$-BKZ on the lattice~$\LL$ with
dimension~$n$. 
Let $e > 0$ such that $\N(\bar R) <  L_Q(1, e)$. Then under the assumption that $\N(\bar R)$ is uniformly distributed over $[1, Q^e]$, the minimal time for the corresponding algorithm to find a $B$-smooth element is
$$
L_Q \left(\frac{1}{3}, (3e)^\frac{1}{3}\right)
$$
reached with $\alpha_B = 2/3$ and $c_B = \left(\frac{e^2}{3}\right)^\frac{1}{3}$.
\end{theorem}

\begin{proof}
The cost of LLL or $\beta$-BKZ being negligible compared to the cost of the smoothness
 test done by ECM, the cost of the algorithm to find a $B$-smooth element is equal 
 to $\mathbf{P}^{-1} \times C$, where $\mathbf{P}$ is the probability of $\bar R$ 
 being $B$-smooth and $C$ is the cost of ECM. The cost of the 
 smoothness test done by ECM is $L_Q(\alpha_B/2, (2c_B \alpha_B)^{1/2})$.
So according to Theorem~\ref{canfield} the cost of our algorithm when $s=0$ (which
is exactly the method proposed in~\cite{guillevic_descent}) is:
$$
L_Q\left(\frac{\alpha_B}{2}, \ (2c_B \alpha_B)^\frac{1}{2}\right)\cdot  L_Q\left(1 - \alpha_B, \ \left(1 - \alpha_B\right) \frac{e}{c_B}\right).
$$
We want to minimize the above quantity. Let's start by minimizing the 
parameter $\max\left(\alpha_B/2, \ 1-\alpha_B\right)$ under the condition 
that this maximum must be lower than~$1/3$. 
Since the condition 
$\left \{ 
\begin{array}{ccc}
\frac{\alpha_B}{2} & \leq & 1/3\\
1-\alpha_B & \leq & 1/3
\end{array}
\right.$
is equivalent to $\alpha_B = 2/3$. 
We conclude that the optimal choice is $\alpha_B = 2/3$. This
is the first value we are looking for. Then, 
the cost of the algorithm becomes $L_Q\left(1/3, (4c_B/3)^{1/2} + e/(3 c_B) \right)$ which is minimal for $c_B = (e^2/3)^{1/3}$. This gives the announced cost. 
\end{proof}

\begin{corollary}[New asymptotic complexities for the individual logarithm step in composite extension degree]
\label{lower_complexity}
Let $p$ be the characteristic of a target finite field, $n$ its composite extension degree, 
$d$ the largest proper divisor of $n$, $f$ 
the polynomial defining the number field for the lift, and $\zeta$ such that~$
\| f \|_{\infty} = p^{\zeta}$. We consider our algorithm (algorithm~\ref{our_algorithm}) where $s$ is set to zero, meaning that no rows or columns are removed from the matrix.
Then the minimal complexity to find a $B$-smooth element in the number
field is:
\[
L_Q\left(\frac{1}{3}, (3e)^\frac{1}{3}\right)
\]
reached with $B = L_Q \left( \frac{2}{3}, \left(\frac{e^2}{3}\right)^\frac{1}{3}\right)$ where 
\begin{itemize}
\item $e = 1 + \zeta - \frac{d}{n} + (3c^{3/2})^{-1}$ if $\alpha = \frac{1}{3}$. This complexity is reached with BKZ only.
\item $e = 1 + \zeta - \frac{d}{n}$ if $ \frac{1}{3} < \alpha < \frac{1}{2}$. This complexity is reached with BKZ only.
\item $e = 1+\zeta - \frac{d}{n}$ if $\alpha = \frac{1}{2}$, reached with BKZ. For the sake of comparison, an algorithm with LLL as in ~\cite{guillevic_descent}
gives $e = 1+\zeta - \frac{d}{n} + \frac{\log(2)}{c^2}$.
\item $e = 1+\zeta - \frac{d}{n}$ if $\frac{1}{2}<\alpha<\frac{2}{3}$. This is reached either
with BKZ or with LLL.
\item $e = 1+\zeta - \frac{d}{n}$ if $\frac{2}{3}\leq\alpha<1$. This is reached
with enumeration, BKZ, or LLL.
\end{itemize}
\end{corollary}

Figure~\ref{fig:sublattices} in the introduction represents 
the complexities given by Corollary~\ref{lower_complexity} when $n$ is even and $\zeta$ is set to $1/2$.

\paragraph{Comparison with previous algorithms.}
While both Algorithm~\cite{guillevic_descent} and Algorithm~\ref{our_algorithm} 
using BKZ have the same 
asymptotic complexity when $1/2 < \alpha \leq 1$,
using BKZ allows to get a lower complexity when $\alpha = 1/2$.
Moreover, \cite{guillevic_descent} does not apply when $1/3 \leq \alpha < 1/2$, and 
to our knowledge, the only previous smoothing algorithm that works in this area 
is the Waterloo~\footnote{
Waterloo algorithm is designed for smoothing in small characteristic finite fields 
but is usable in this area too.} algorithm~\cite{C:BlaMulVan84}.
The asymptotic complexity of
this method when $1/3 \leq \alpha < 1/2$ is in $L_Q \left(1/3, \left(3(2+\zeta)\right)^{1/3}\right)$.
Hence our 
Algorithm \ref{our_algorithm} with BKZ is faster, as it has an asymptotic complexity in $L_Q \left(1/3, \left(3(1 +\zeta - d/n)\right)^{1/3}\right)$ in the same area.
Nevertheless, the Waterloo technique applies on any extension degree 
whereas Algorithm~\ref{our_algorithm} applies only on composite extension degrees.

\begin{remark}
When we get a new smaller $e$ value in the above corollary, we have a double gain. Indeed, it provides us with both a smaller complexity for
the initial splitting step -- we get a smooth element faster --
and a smaller smoothness bound -- the obtained element is more smooth, thus better for the descent step.
\end{remark}

\begin{example} Let us target a finite field with even extension degree~$n$
and characteristic~$p$.
Construct the number fields and the target finite field thanks to a 
polynomial selection
 that guaranties $\zeta = 0$ and $\deg(f)=n$. The Conjugation method is a good 
 example of such a selection. Theses parameters lead to $e = 1/2$
 because $d/n=2$.
 Then the complexity of the initial splitting step brought 
 by our algorithm using BKZ is:
\[L_Q\left( \frac{1}{3}, \left(\frac{3}{2}\right)^{\frac{1}{3}}\right),\]
where $(3/2)^{1/3} \approx 1.14$.
This value for the complexity is reached for any $p > L_Q(1/3)$. For the
sake of comparison, we recall that the complexity brought by the Waterloo algorithm in medium characteristic 
finite fields  is  $L_Q(1/3, 1.82)$.
\end{example}

\subsection{Combining the sublattice method with BKZ or enumeration}
\label{subsec:BKZ_method_s}

In this section we present  a mix of the two previous methods: 
we study the behavior 
  of BKZ or an enumeration algorithm on a sublattice~$\LL'$, namely 
  we set $s>0$.
We look at Algorithm~$\ref{our_algorithm}$ where the reduction algorithm is $\beta$-BKZ, with $\beta = (n-s)^{(3 (1-\alpha))^{-1}}$ as in Paragraph~\ref{choice-of-beta} if $\alpha < 2/3$. 
If $\alpha \geq 2/3$ then we use an enumeration algorithm on the sublattice
derived from $\LL$ by deleting $s$ rows and colomns. 
As in Paragraph~\ref{subsec:sublattice} we study the optimal choice of $s$ 
over $[ 0, d-2]$ that minimizes the norms of the candidates~$\bar R$.

\paragraph{BKZ on sublattices.}
In order to do so, using $\beta$-BKZ and Theorems~\ref{bound-in-number-field} 
and~\ref{BKZ-bound} we get an upper bound on $\N(\bar R)$ as a function of $s$. Recall that the degree of $\bar R$ is upper bounded by $n-s-1$. We have:
\[
\N(\bar R) = O \left( 2^n \beta^{n \left( \frac{n-s-1}{2(\beta-1)} + \frac{3}{2} \right)} \ p^{n\frac{n-d}{n-s} + \zeta (n-s-1)}\right).
\]
Again our aim is to find the integer $s$ in $[ 0, d-2 ]$  that minimizes the function $h_{\hbox{\tiny BKZ}} : s \mapsto 2^n \beta^{n \left( \frac{n-s-1}{2(\beta-1)} + \frac{3}{2} \right)} \ p^{n\frac{n-d}{n-s} + \zeta (n-s-1)}.$
Let us write $\tilde s_1 = n - (2 (n-d) n \log(p) \log(\beta-1))^{1/2} \cdot (\log(\beta) + 2 \zeta \log(p)(\beta-1))^{-1/2} $, and $s_2$ be an integer such that $\tilde s_2 > n$. 
A simple analysis gives the following variation table for $h_{\hbox{\tiny BKZ}}$: 
\begin{center}
\begin{tikzpicture}[scale=0.7]
   \tkzTabInit{$s$ / 1 , $h_{\hbox{\tiny BKZ}}$ / 1}{,$\tilde s_1$, $\tilde s_2$,}
   \tkzTabVar{+/, -/, +/, -/}
\end{tikzpicture}
\end{center}
As in Paragraph~\ref{subsec:sublattice}, we deduce the following result
 that explicits where the function $h_{\hbox{\tiny BKZ}}$ is minimum over the integers 
between $0$ and $d-2$:

\begin{lemma} Let $\tilde s_1 = n - (2 (n-d) n \log(p) \log(\beta-1))^{1/2} \cdot (\log(\beta) + 2 \zeta \log(p)(\beta-1))^{-1/2} $. Then the number~$s_{opt}$ of rows and columns
to delete is given by the following cases:
\begin{enumerate}
\item \emph{If $\tilde s_1 < 0$}, then $s_{opt} = 0$.
\item \emph{If $0 \leq \tilde s_1 \leq d-2$}, then $s_{opt} = \lfloor \tilde s_1 \rfloor$ or $s_{opt} = \lceil \tilde s_1 \rceil$.
\item \emph{If $\tilde s_1 > d-2$}, then $s_{opt} = d-2$.
\end{enumerate}
\end{lemma}
Let us  give a simple condition on $\zeta$ that ensures that the optimal sublattice to 
choose is a strict sublattice, and thus for such values of $\zeta$, we expect that this 
new algorithm outperforms 
all the smoothness algorithms mentioned previously. Our algorithm outputs
better candidates for the norms in the number field as soon as $\tilde s_1 \geq 1$.
Since 
$\tilde s_1 \geq 1\iff (n-1)^2 \geq \frac{2 (n-d) n  \log(\beta-1) \log p }{\log \beta 
+ 2 \zeta (\beta-1)  \log p }$, which is equivalent to have:
\[ \zeta \geq \frac{n \cdot (n-d)}{(n-1)^2} \cdot \frac{\log(\beta-1)}{\beta - 1} 
- \frac{\log \beta}{2 (\beta-1)  \log p}. \]

This condition for even extension degrees can be written as:
\[\zeta \geq \frac{1}{2} \left(\frac{n}{n-1}\right)^2 \cdot \frac{\log(\beta-1)}{\beta-1} - \frac{\log \beta}{2  (\beta - 1) \log p}.\]

\begin{example} We focus on the family of finite fields with fixed extension degree~$n=24$. 
Choosing $\beta = 6$ and looking at JLSV$_1$ for the polynomial selection, we see that: $\zeta = 0.5$ in one hand,
and $\frac{1}{2} \left(\frac{n}{n-1}\right)^2 \frac{\log(\beta-1)}{\beta-1} \approx 0.15$
 in the other hand, meaning that these parameters offer a convenient settings for our improvements. Similarly, choosing $n=12$, $\beta = 3$, and looking at JLSV$_1$, we have $\frac{1}{2} \left(\frac{n}{n-1}\right)^2 \frac{\log(\beta-1)}{\beta-1} \approx 0.21$
 which is a nice setting too.
 \end{example}
 
Even when the best choice of s is greater than 0, we still get norms in $L_Q(1, 1 + \zeta - d/n)$ as the ones we get when setting $s$ to $0$.
An optimal $s_{opt}>0$ means that we get smaller, yet asymptotically equivalent norms. In this sense, considering sublattices does not allow to lower the smoothing step 
asymptotic complexity.

\paragraph{Enumeration on sublattices.}
When an enumeration algorithm is used, the bound on the norm of the output is $ (n-s)^{n/2} p^{n\frac{n-d}{n-s} + \zeta (n-s-1)}$. Since we deal with large characteristic finite fields, any polynomial selection outputs polynomials of infinite norm smaller than $p^{1/2}$, thus we can assume $\zeta \leq 1/2$. Under this assumption, the bound above, as a function of s, is an increasing function, it reaches its minimum over the integers $[ 0, d-2 ]$ at $s=0$. We conclude that in large characteristic finite fields, when using an enumeration algorithm in Algorithm~\ref{our_algorithm}, it is asymptotically useless to decrease the lattice dimension.

\section{Lower practical norms}
\label{sec:practical}
As in most of discrete logarithm algorithms, we cannot deduce 
the behavior of our method on practical sizes by only looking at 
the major improvement on the asymptotic complexity. To tackle this question 
 we present in this section practical results obtained with our implementation
of Algorithm~\ref{our_algorithm}. This implementation including the finite field construction is given in \cite{Code:smoothness}. 
On the examples and sizes we have looked at, BKZ did not lead to
real important improvements for the norms with respect to LLL. For this reason
we present only our experiments with LLL to perform the lattice reductions
on sublattices of various sizes.

One run of our implementation takes as input a random target $T$ in a finite field 
of composite extension degree, a relevant and compatible number field~$\K$ 
and a parameter $s \in [ 0, d-2 ]$ and creates an 
element~$\bar R$ in~$\K$ to be tested for $B$-smoothness. This method is not applicable for $n=4$ but starts with a potential effect as soon as $d>2$, \textit{i.e.} $n \geq 6$. Note that whenever $s$ is set to $0$, 
then our implementation is an implementation of Guillevic's 
algorithm~\cite{guillevic_descent}, without the smoothness test.

\subsubsection{Target finite fields.} We consider 148 different finite fields with 
composite extension degrees varying from 4 to 50. Half of them have a $ 460$
to $500$-bit size
while the others
have a 2050 to 2080-bit size.
In order not to make the text more cumbersome, we use in the sequel
the term $500$-bit size and $2048$-bit size to refer to these two different
families of fields.
Indeed, for each composite degree $4 \leq n \leq 50$, 
the characteristic is set to the first prime larger than $2^{500//n}$ 
(resp.  $2^{2048//n + 1}$).
Note that we conducted experiments
with 700 and 1024-bit sizes too, but the results are similar and for the
sake of simplicity we do not 
detail these experiments here.

Each field~$\F_{p^n}$ is built alongside with a number field $\K_f = \Q[X]/(f)$ 
where $f$ is one of the polynomial given by  the JLSV$_1$ polynomial selection.
Thus we have $\deg(f)=n$ and $\|f\|_{\infty} = p^{1/2}$.
For each finite field, we ran an optimization code based on the alpha value \cite{guillevic:hal-02263098} and coefficients sizes to select the polynomials. 
The polynomials were selected among 100 pairs produced by the JLSV$_1$ polynomial selection.
The code for selecting the polynomials as well as the polynomials can be found at the GitLab repository \cite{Code:smoothness}.

\paragraph{Other polynomial selection methods.}
Other experiments not provided here show that our algorithm produces practical improvements when the coefficients of the polynomial that defines the finite field are sufficiently large. For instance, we do not manage to reduce the norms by more than 10 bits by using the Conjugation method.

\subsubsection{Target elements.}
In each finite field we randomly draw $1000$ elements that become our $1000$ targets. 
Each element~$T$ is given as an input for two algorithms: we note $R_1$ the output in the number field of Guillevic's one, and $R_2$ the output of our Algorithm~\ref{our_algorithm}. For each field we compute the mean of the norms 
in~$\K$ of all lifted targets, the mean of the norms in~$\K$ of all $R_1$, and the mean of the norms in~$\K$ of all~$R_2$. Auxiliary data and in particular 
these means are reported in Appendix~\ref{Appendix:Data}. Moreover, our implementation of algorithm~\ref{our_algorithm} can be found at the GitLab repository \cite{Code:smoothness}.

\subsubsection{Theoretical optimal choices versus practical experiments.}
For each finite field, we test several values of $s$. This allows to see that we are 
not yet experiencing asymptotic phenomenons, as the theoretical~$s$ 
given in Theorem~\ref{theo:setting_s} and the best 
practical ones differ. Again, these theoretical $s$
and practical good ones are in Appendix~\ref{Appendix:Data}.
For instance for the $476$-bit field~$\F_{p^{28}}$ the theoretical optimal $s$ is equal to $6$ whereas in practice $s=3$ gives good results.

\subsubsection{Lower norms in the number field.}
The results on 500-bit finite fields and the 2048-bit ones are
each presented with $4$ graphics: 
\begin{itemize}
\item Figures~\ref{fig_norms_500} and \ref{fig_norms_2048} show the
norms in~$\K$ of the lifted elements, of Guillevic's candidates, and ours, all as a function of~$n$.
\item Figures~\ref{fig_diff_500} and \ref{fig_diff_2048} present the benefit of our
method on the bitsizes of the norms in~$\K$ as a function of~$n$.
\item Figures~\ref{fig_diff_d_500} and \ref{fig_diff_d_2048} present the benefit of our
method on the bitsizes of the norms in~$\K$, but as
a function of $d$ the largest divisor of $n$.
\end{itemize}

All theses graphs do support our previous analysis.
Algorithm~\ref{our_algorithm} outputs elements of smaller norms in the number field than those output by~\cite{guillevic_descent}. For instance in the $500$-bit finite field of extension $16$ (resp. $48$) Algorithm~\ref{our_algorithm} allows to get elements to be tested for $B$-smoothness of size $6$ bits smaller  (resp. $36$ bits) than those output by~\cite{guillevic_descent}.
In the $2048$-bit finite field of extension degree~$32$ (resp. $50$), Algorithm~\ref{our_algorithm} allows to get elements to be tested for $B$-smoothness of size $11$ bits smaller (resp. $25$ bits) than those output by~\cite{guillevic_descent}.

\subsubsection{Higher Euclidean norms.} Figures~\ref{fig_eulidean_500} and \ref{fig_eulidean_2048} show the difference of sizes in basis~$2$ between the average Euclidean norms of Guillevic's candidates
 and the average Euclidean norms of candidates from Algorithm~\ref{our_algorithm}, as a function of the extension degree~$n$.
As expected,
the outputs of LLL  performed on our sublattices have greater Euclidean norms 
than those output by LLL on the original full lattice.

\subsubsection{The largest divisor effect.}
Moreover, one important remark is illustrated thanks to Figures~\ref{fig_diff_d_500} and \ref{fig_diff_d_2048}. 
Here we see that the higher $d$ is, the better our method performs 
with regard to~\cite{guillevic_descent}. 
This is not surprising as when $d$ increases, the set of choices
for the parameter~$s$ increases, and the degree of the output
decreases with~$s$.

\subsubsection{Improvement in the probability of smoothness.}
Let us look closely at two examples and compute the gain
we get in term of smoothness probability. 
First we look at the $500$-bit finite field~$\F_{p^{16}}$.
\cite{guillevic_descent} allows to get $501$-bit norms  and our algorithm 
gives $495$-bit norms. Let us set $B = 2^{35}$ 
for the smoothness bound\footnote{This is the value chosen in the 521-bit TNFS record 
on $\F_{p^6}$~\cite{DBLP:conf/asiacrypt/MicheliGP21}}. Using the dickman\_rho function 
implemented in sage, we get that the probability of \cite{guillevic_descent}'s output to be $B$-smooth is about $1.30 \times 10^{-18}$, and the one of our output 
is $3.65 \times 10^{-18}$. Our output is twice as likely to be $B$-smooth. 
Another example with the $2048$-bit finite field~$\F_{p^{50}}$.
We get using \cite{guillevic_descent}'s algorithm $2144$-bit norms whereas using Algorithm~\ref{our_algorithm} the norms have sizes around $2119$~bits.
Let us set $B = 2^{80}$. 
In this case our outputs are 4.6 times as likely to be $B$-smooth.

\begin{figure}[H]
\begin{minipage}[c]{0.5\textwidth}
\includegraphics[scale=0.4]{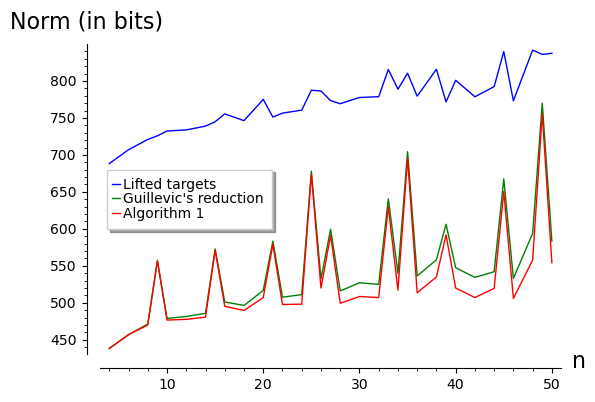}
\caption{Average norms in the number field of the lifted targets, of Guillevic's candidates
 and of candidates from Algorithm~\ref{our_algorithm}, as a function
 of the extension degree~$n$. Experiments run on approximately 500-bit finite fields.}
\label{fig_norms_500}
\end{minipage}
$\quad$
\begin{minipage}[c]{0.5\textwidth}
\includegraphics[scale=0.4]{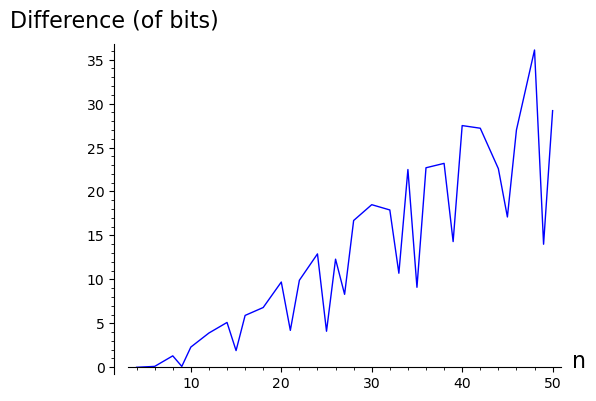}
\caption{Difference of sizes in basis 2 between the average norms of Guillevic's candidates
 and the average norms of candidates from Algorithm~\ref{our_algorithm}, as a function of the extension degree~$n$. Experiments run on  approximately 500-bit finite fields.}
\label{fig_diff_500}
\end{minipage}

\end{figure}

\vspace{-1.5cm}
\begin{figure}[H]

\begin{minipage}[c]{0.5\textwidth}
\includegraphics[scale=0.4]{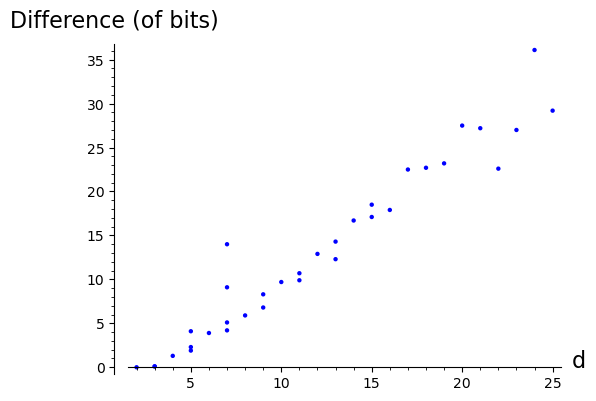}
\caption{Difference of sizes in basis 2 between the average norms of Guillevic's candidates
 and the average norms of candidates from Algorithm~\ref{our_algorithm}, as a function
 of $d$ the greatest proper divisor of $n$. Experiments run on  approximately 500-bit finite fields.}
\label{fig_diff_d_500}
\end{minipage}
$\quad$
\begin{minipage}[c]{0.5\textwidth}
\includegraphics[scale=0.4]{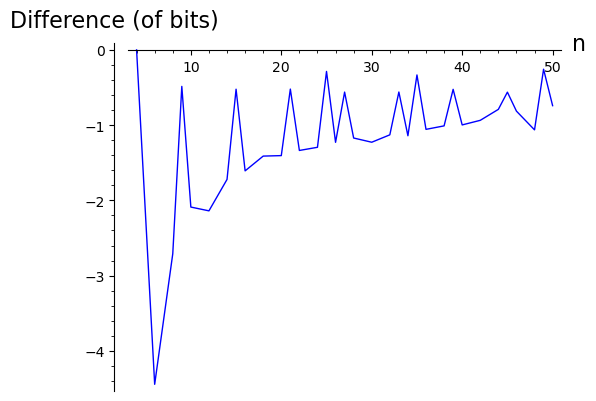}
\caption{Difference of sizes in basis 2 between the average Euclidean norms of Guillevic's candidates
 and the average Euclidean norms of candidates from Algorithm~\ref{our_algorithm}, as a function of the extension degree~$n$. Experiments run on  approximately 500-bit finite fields.}
\label{fig_eulidean_500}
\end{minipage}

\end{figure}

\begin{figure}[H]

\begin{minipage}[c]{0.5\textwidth}
\includegraphics[scale=0.4]{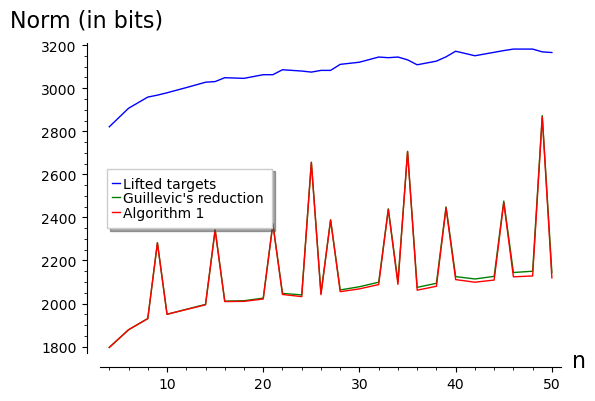}
\caption{Average norms in the number field of the lifted targets, of Guillevic's candidates
 and of candidates from Algorithm~\ref{our_algorithm}, as a function
 of the extension degree~$n$. Experiments run on  approximately 2048-bit finite fields.}
\label{fig_norms_2048}
\end{minipage}
$\quad$
\begin{minipage}[c]{0.5\textwidth}
\includegraphics[scale=0.4]{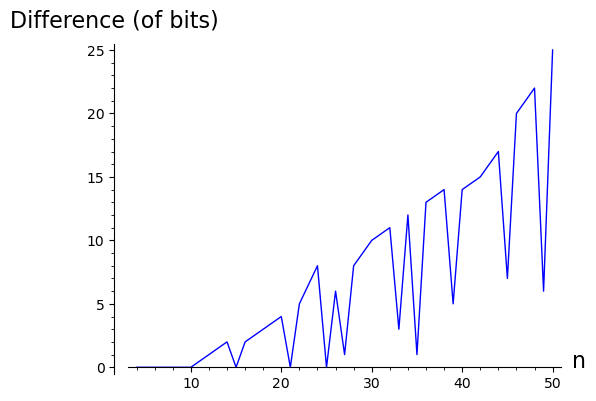}
\caption{Difference of sizes in basis 2 between the average norms of Guillevic's candidates
 and the average norms of candidates from Algorithm~\ref{our_algorithm}, as a function of the extension degree~$n$. Experiments run on approximately 2048-bit finite fields.}
\label{fig_diff_2048}
\end{minipage}
\end{figure}

\begin{figure}[H]

\begin{minipage}[c]{0.5\textwidth}
\includegraphics[scale=0.4]{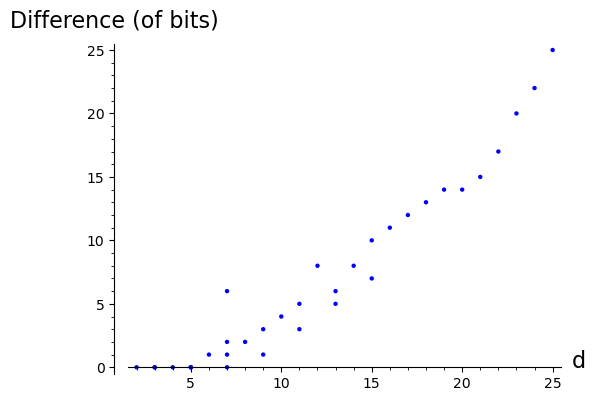}
\caption{Difference of sizes in basis 2 between the average norms of Guillevic's candidates
 and the average norms of candidates from Algorithm~\ref{our_algorithm}, as a function
 of $d$ the greatest proper divisor of $n$. Experiments run on approximately 2048-bit finite fields.}
\label{fig_diff_d_2048}
\end{minipage}
$\quad$
\begin{minipage}[c]{0.5\textwidth}
\includegraphics[scale=0.4]{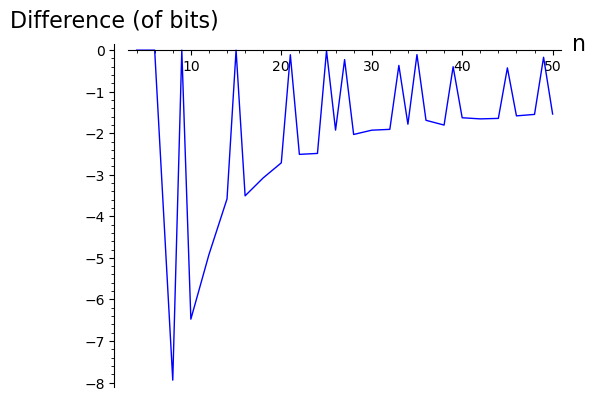}
\caption{Difference of sizes in basis 2 between the average Euclidean norms of Guillevic's candidates
 and the average Euclidean norms of candidates from Algorithm~\ref{our_algorithm}, as a function of the extension degree~$n$. Experiments run on approximately 2048-bit finite fields.}
\label{fig_eulidean_2048}
\end{minipage}

\end{figure}

\section*{Conclusion}
We proved that using BKZ reduction instead of LLL lowers the individual logarithm complexity in the lower half of the medium characteristic range.

In addition, experiments show that using sublattices to perform the smoothness step in the number field sieve can outperform the existing technique of using the whole lattice. This new technique outperforms the later when the composite extension degree is sufficiently large and the coefficients of the polynomial constructing the number field are large enough. 
For instance, these two conditions are fulfilled when dealing with medium characteristic finite fields and using the JLSV$_1$ polynomial selection. This set up is relevant since the JLSV$_1$ polynomial selection is both adapted in theory and in practice for TNFS and is well adapted for MNFS in the medium characteristic case, especially when one asks for a symmetric diagram. Such setting can be very useful for the MexTNFS variant in order to get many number fields of the same quality.

\section*{Declaration}
\begin{itemize}
\item \textbf{Funding:} Haetham Al Aswad is funded by French 
Ministry of Army  - AID Agence de l'Innovation de D\'{e}fense.  
C\'{e}cile Pierrot did not receive support from any organization 
for the submitted work.
\item \textbf{Financial interests:} The authors declare they have no financial interests.
\end{itemize}

The datasets generated during the  study are available in the "Smoothing step in NFS for composite extension degree finite fields" git lab repository, \cite{Code:smoothness}.

\section*{Acknowledgment}

This version of the article has been accepted for publication, after 
peer
review but is not the Version of Record and does 
not reflect post-acceptance
improvements, or any corrections. The Version of Record is 
available online at:
http://dx.doi.org/10.1007/s10623-023-01282-w. Use of this Accepted Version is 
subject to the publisher’s Accepted
Manuscript terms of use: \\ https://www.springernature.com/gp/open-
research/policies/acceptedmanuscript-terms

\bibliographystyle{alpha}
\bibliography{abbrev3,crypto,tnfs_bib}

\newcommand{\etalchar}[1]{$^{#1}$}
\begin{thebibliography}{BFHMV84}

\bibitem[AP22]{Code:smoothness}
Haetham~Al Aswad and C\'ecile Pierrot.
\newblock Smoothness step in {NFS} for composite extenion finite fields.
  https://gitlab.inria.fr/halaswad/smoothing-step-in-nfs-for-composite-extension-degree-finite-fields.
\newblock GitLab, 2022.

\bibitem[BF01]{C:BonFra01}
Dan Boneh and Matthew~K. Franklin.
\newblock Identity-based encryption from the {Weil} pairing.
\newblock In Joe Kilian, editor, {\em CRYPTO~2001}, volume 2139 of {\em
  {LNCS}}, pages 213--229. Springer, Heidelberg, August 2001.

\bibitem[BFHMV84]{blakeFMV}
Ian Blake, R.~Fuji-Hara, R.~Mullin, and S.~Vanstone.
\newblock Computing logarithms in finite fields of characteristic two.
\newblock {\em Siam Journal on Algebraic and Discrete Methods}, 5, 06 1984.

\bibitem[BGG{\etalchar{+}}20]{C:BGGHTZ20}
Fabrice Boudot, Pierrick Gaudry, Aurore Guillevic, Nadia Heninger, Emmanuel
  Thom{\'e}, and Paul Zimmermann.
\newblock Comparing the difficulty of factorization and discrete logarithm: {A}
  240-digit experiment.
\newblock In Hovav Shacham and Alexandra Boldyreva, editors, {\em CRYPTO~2020,
  Part~II}, {LNCS}, pages 62--91. Springer, Heidelberg, August 2020.

\bibitem[BGGM15a]{recordfp4}
Razvan Barbulescu, Pierrick Gaudry, Aurore Guillevic, and Fran{\c{c}}ois
  Morain.
\newblock {DL} record computation in $\mathbb{F}_{p^4}$ of $392$ bits, 2015.
\newblock
  \url{http://www.lix.polytechnique.fr/~guillevic/docs/guillevic-catrel15-talk.pdf}.

\bibitem[BGGM15b]{EC:BGGM15}
Razvan Barbulescu, Pierrick Gaudry, Aurore Guillevic, and Fran{\c c}ois Morain.
\newblock Improving {NFS} for the discrete logarithm problem in non-prime
  finite fields.
\newblock In Elisabeth Oswald and Marc Fischlin, editors, {\em EUROCRYPT~2015,
  Part~I}, volume 9056 of {\em {LNCS}}, pages 129--155. Springer, Heidelberg,
  April 2015.

\bibitem[BGJT14]{EC:BGJT14}
Razvan Barbulescu, Pierrick Gaudry, Antoine Joux, and Emmanuel Thom{\'e}.
\newblock A heuristic quasi-polynomial algorithm for discrete logarithm in
  finite fields of small characteristic.
\newblock In Phong~Q. Nguyen and Elisabeth Oswald, editors, {\em
  EUROCRYPT~2014}, volume 8441 of {\em {LNCS}}, pages 1--16. Springer,
  Heidelberg, May 2014.

\bibitem[BGK15]{AC:BarGauKle15}
Razvan Barbulescu, Pierrick Gaudry, and Thorsten Kleinjung.
\newblock The tower number field sieve.
\newblock In Tetsu Iwata and Jung~Hee Cheon, editors, {\em ASIACRYPT~2015,
  Part~II}, volume 9453 of {\em {LNCS}}, pages 31--55. Springer, Heidelberg,
  November~/~December 2015.

\bibitem[BLS01]{AC:BonLynSha01}
Dan Boneh, Ben Lynn, and Hovav Shacham.
\newblock Short signatures from the {Weil} pairing.
\newblock In Colin Boyd, editor, {\em ASIACRYPT~2001}, volume 2248 of {\em
  {LNCS}}, pages 514--532. Springer, Heidelberg, December 2001.

\bibitem[BMV84]{C:BlaMulVan84}
Ian~F. Blake, Ronald~C. Mullin, and Scott~A. Vanstone.
\newblock Computing logarithms in {$\text{GF}(2^n)$}.
\newblock In G.~R. Blakley and David Chaum, editors, {\em CRYPTO'84}, volume
  196 of {\em {LNCS}}, pages 73--82. Springer, Heidelberg, August 1984.

\bibitem[BP14]{barbulescu:hal-00952610}
Razvan Barbulescu and C{\'e}cile Pierrot.
\newblock {The Multiple Number Field Sieve for Medium and High Characteristic
  Finite Fields}.
\newblock {\em {LMS Journal of Computation and Mathematics}}, 17:230--246,
  2014.

\bibitem[CEP83]{canfield1983problem}
E.~Rodney Canfield, Paul Erd{\"o}s, and Carl Pomerance.
\newblock On a problem of {O}ppenheim concerning “factorisatio numerorum”.
\newblock {\em Journal of Number Theory}, 17(1):1--28, 1983.

\bibitem[FP85]{Fincke1985ImprovedMF}
Ulrich Fincke and Michael~E. Pohst.
\newblock Improved methods for calculating vectors of short length in a
  lattice.
\newblock {\em Mathematics of Computation}, 1985.

\bibitem[GKZ14]{C:GraKleZum14}
Robert Granger, Thorsten Kleinjung, and Jens Zumbr{\"a}gel.
\newblock Breaking `128-bit secure' supersingular binary curves - (or how to
  solve discrete logarithms in {$\mathbb{F}_{2^{4 \cdot 1223}}$} and
  {$\mathbb{F}_{2^{12 \cdot 367}}$}).
\newblock In Juan~A. Garay and Rosario Gennaro, editors, {\em CRYPTO~2014,
  Part~II}, volume 8617 of {\em {LNCS}}, pages 126--145. Springer, Heidelberg,
  August 2014.

\bibitem[Gr{\'{e}}]{dldatabase}
Laurent Gr{\'{e}}my.
\newblock Computations of discrete logarithms sorted by date.
\newblock \url{https://dldb.loria.fr/}.

\bibitem[GS21]{guillevic:hal-02263098}
Aurore Guillevic and Shashank Singh.
\newblock {On the Alpha Value of Polynomials in the Tower Number Field Sieve
  Algorithm}.
\newblock {\em {Mathematical Cryptology}}, 1(1):39, 2021.

\bibitem[Gui15]{AC:Guillevic15}
Aurore Guillevic.
\newblock Computing individual discrete logarithms faster in {$\text{GF}(p^n)$}
  with the {NFS}-{DL} algorithm.
\newblock In Tetsu Iwata and Jung~Hee Cheon, editors, {\em ASIACRYPT~2015,
  Part~I}, volume 9452 of {\em {LNCS}}, pages 149--173. Springer, Heidelberg,
  November~/~December 2015.

\bibitem[Gui19]{guillevic_descent}
Aurore Guillevic.
\newblock {Faster individual discrete logarithms in finite fields of composite
  extension degree}.
\newblock {\em {Mathematics of Computation}}, 88(317):1273--1301, January 2019.

\bibitem[HPS11]{C:HanPujSte11}
Guillaume Hanrot, Xavier Pujol, and Damien Stehl{\'e}.
\newblock Analyzing blockwise lattice algorithms using dynamical systems.
\newblock In Phillip Rogaway, editor, {\em CRYPTO~2011}, volume 6841 of {\em
  {LNCS}}, pages 447--464. Springer, Heidelberg, August 2011.

\bibitem[JLSV06]{C:JLSV06}
Antoine Joux, Reynald Lercier, Nigel Smart, and Frederik Vercauteren.
\newblock The number field sieve in the medium prime case.
\newblock In Cynthia Dwork, editor, {\em CRYPTO~2006}, volume 4117 of {\em
  {LNCS}}, pages 326--344. Springer, Heidelberg, August 2006.

\bibitem[Jou04]{DBLP:journals/joc/Joux04}
Antoine Joux.
\newblock A one round protocol for tripartite {D}iffie-{H}ellman.
\newblock {\em Journal of Cryptology}, 17(4):263--276, 2004.

\bibitem[JP14]{PAIRING:JouPie13}
Antoine Joux and C{\'e}cile Pierrot.
\newblock The special number field sieve in {$\mathbb{F}_{p^n}$} - application
  to pairing-friendly constructions.
\newblock In Zhenfu Cao and Fangguo Zhang, editors, {\em PAIRING 2013}, volume
  8365 of {\em {LNCS}}, pages 45--61. Springer, Heidelberg, November 2014.

\bibitem[Kan87]{kannan87}
Ravi Kannan.
\newblock Minkowski's convex body theorem and integer programming.
\newblock {\em Mathematics of Operations Research}, 12(3):415–440, aug 1987.

\bibitem[KB16]{C:KimBar16}
Taechan Kim and Razvan Barbulescu.
\newblock Extended tower number field sieve: {A} new complexity for the medium
  prime case.
\newblock In Matthew Robshaw and Jonathan Katz, editors, {\em CRYPTO~2016,
  Part~I}, volume 9814 of {\em {LNCS}}, pages 543--571. Springer, Heidelberg,
  August 2016.

\bibitem[KJ17]{PKC:KimJeo17}
Taechan Kim and Jinhyuck Jeong.
\newblock Extended tower number field sieve with application to finite fields
  of arbitrary composite extension degree.
\newblock In Serge Fehr, editor, {\em PKC~2017, Part~I}, volume 10174 of {\em
  {LNCS}}, pages 388--408. Springer, Heidelberg, March 2017.

\bibitem[KW19]{EPRINT:KleWes19}
Thorsten Kleinjung and Benjamin Wesolowski.
\newblock Discrete logarithms in quasi-polynomial time in finite fields of
  fixed characteristic.
\newblock Cryptology ePrint Archive, Report 2019/751, 2019.
\newblock \url{https://eprint.iacr.org/2019/751}.

\bibitem[LLL82]{LLL}
A.~K. Lenstra, H.~W. Lenstra, and L.~Lovasz.
\newblock {Factoring Polynomials with Rational Coefficients}.
\newblock {\em Mathematische Annalen}, 261:515--534, 1982.

\bibitem[MGP21]{DBLP:conf/asiacrypt/MicheliGP21}
Gabrielle~De Micheli, Pierrick Gaudry, and C{\'{e}}cile Pierrot.
\newblock Lattice enumeration for tower {NFS:} {A} 521-bit discrete logarithm
  computation.
\newblock In Mehdi Tibouchi and Huaxiong Wang, editors, {\em Advances in
  Cryptology - {ASIACRYPT} 2021 - 27th International Conference on the Theory
  and Application of Cryptology and Information Security, Singapore, December
  6-10, 2021, Proceedings, Part {I}}, volume 13090 of {\em Lecture Notes in
  Computer Science}, pages 67--96. Springer, 2021.

\bibitem[Mil04]{Miller_pairing}
Victor Miller.
\newblock The {W}eil pairing, and its efficient calculation.
\newblock {\em Journal of Cryptology}, 17:235--261, 2004.

\bibitem[MS20]{DBLP:journals/ffa/MukhopadhyayS20}
Madhurima Mukhopadhyay and Palash Sarkar.
\newblock Faster initial splitting for small characteristic composite extension
  degree fields.
\newblock {\em Finite Fields and Their Applications}, 62:101629, 2020.

\bibitem[MSST22]{Mukhopadhyay2022}
Madhurima Mukhopadhyay, Palash Sarkar, Shashank Singh, and Emmanuel Thom\'{e}.
\newblock New discrete logarithm computation for the medium prime case using
  the function field sieve.
\newblock {\em Advances in Mathematics of Communications}, 16(3):449--464,
  2022.

\bibitem[MV10]{STOC:MicVou10}
Daniele Micciancio and Panagiotis Voulgaris.
\newblock A deterministic single exponential time algorithm for most lattice
  problems based on voronoi cell computations.
\newblock In Leonard~J. Schulman, editor, {\em 42nd ACM STOC}, pages 351--358.
  {ACM} Press, June 2010.

\bibitem[MW15]{SODA:MicWal15}
Daniele Micciancio and Michael Walter.
\newblock Fast lattice point enumeration with minimal overhead.
\newblock In Piotr Indyk, editor, {\em 26th SODA}, pages 276--294. {ACM-SIAM},
  January 2015.

\bibitem[MW16]{EC:MicWal16}
Daniele Micciancio and Michael Walter.
\newblock Practical, predictable lattice basis reduction.
\newblock In Marc Fischlin and Jean-S{\'{e}}bastien Coron, editors, {\em
  EUROCRYPT~2016, Part~I}, volume 9665 of {\em {LNCS}}, pages 820--849.
  Springer, Heidelberg, May 2016.

\bibitem[Pie15]{EC:Pierrot15}
C{\'e}cile Pierrot.
\newblock The multiple number field sieve with conjugation and generalized
  {Joux}-{Lercier} methods.
\newblock In Elisabeth Oswald and Marc Fischlin, editors, {\em EUROCRYPT~2015,
  Part~I}, volume 9056 of {\em {LNCS}}, pages 156--170. Springer, Heidelberg,
  April 2015.

\bibitem[Sch00]{shiro}
Oliver Schirokauer.
\newblock {Using Number Fields to Compute Logarithms in Finite Fields}.
\newblock {\em Mathematics of Computation}, 69:1267--1283, 2000.

\bibitem[SE94]{DBLP:journals/mp/SchnorrE94}
Claus{-}Peter Schnorr and M.~Euchner.
\newblock Lattice basis reduction: Improved practical algorithms and solving
  subset sum problems.
\newblock {\em Mathematical Programming}, 66:181--199, 1994.

\bibitem[SS16]{AC:SarSin16}
Palash Sarkar and Shashank Singh.
\newblock A general polynomial selection method and new asymptotic complexities
  for the tower number field sieve algorithm.
\newblock In Jung~Hee Cheon and Tsuyoshi Takagi, editors, {\em ASIACRYPT~2016,
  Part~I}, volume 10031 of {\em {LNCS}}, pages 37--62. Springer, Heidelberg,
  December 2016.

\bibitem[SS19]{PalashSarkar2019AdvancesinMathematicsofCommunications}
Palash Sarkar and Shashank Singh.
\newblock A unified polynomial selection method for the (tower) number field
  sieve algorithm, 2019.

\bibitem[TPM]{TPM2019}
Trusted platform module.
\newblock
  \url{https://trustedcomputinggroup.org/resource/tpm-library-specification/}.
\newblock Latest Version Nov. 2019.

\bibitem[{Wie}86]{wiedemann}
Douglas~H. {Wiedemann}.
\newblock {Solving Sparse Linear Equations over Finite Fields}.
\newblock {\em IEEE Transactions on Information Theory}, 32(1):54--62, 1986.

\end{thebibliography}

\appendix

\section{Example}
\label{Appendix:Example}

We give a concrete example to better understand 
Algorithm~\ref{our_algorithm} and to see 
how decreasing the degree while allowing larger coefficients 
can result is smaller norms. Take the finite field $\F_{p^{28}}$ of size $476$ bits where $p = 131101$.

\paragraph{Construction of finite field and number fields: }
After running JLSV$_1$ polynomial selection to find 100 pairs of suitable polynomials. We choose the pair with the highest score for a notion of score based on the alpha value \cite{guillevic:hal-02263098} and the coefficient sizes. The code to select the pair of polynomials can be found at~\cite{Code:smoothness}.

$f_1(X) = X^{28} + 349X^{27} + 348X^{26} + 1040X^{25} + 349X^{24} + 348X^{23} + 1040X^{22} + 1040X^{21} + 695X^{20} + 1041X^{19} + 695X^{18} + 347X^{17} + 349X^{16} + 347X^{15} + 348X^{14} + 694X^{13} + 1039X^{12} + 348X^{11} + 347X^{10} + 348X^9 + 1039X^8 + 347X^7 + 695X^6 + 1041X^5 + 349X^4 + 1039X^3 + 347X^2 + 1041X + 349$.

$f_2(X) = -379X^{28} - 1170X^{27} - 791X^{26} - 857X^{25} - 1170X^{24} - 791X^{23} - 857X^{22} - 857X^{21} - 1203X^{20} - 1236X^{19} - 1203X^{18} - 412X^{17} - 1170X^{16} - 412X^{15} - 791X^{14} - 824X^{13} - 478X^{12} - 791X^{11} - 412X^{10} - 791X^9 - 478X^8 - 412X^7 - 1203X^6 - 1236X^5 - 1170X^4 - 478X^3 - 412X^2 - 1236X - 1170$. 

Moreover, $f_1$ is also irreducible in $\F_p[X]$, thus $\F_{p^{28}}$ is represented as:
\[\F_p[X]/(f_1) := \F_p(\alpha).\]
Since $f_1$ has smaller coefficients than $f_2$, it is natural to perform the smoothing step in $\K = \Q[X]/(f_1) := \Q(x)$. Denote by $\N$ the norm defined in $\K$ and for any element $Y$ in $\F_p(\alpha)$, $\bar{Y}$ denotes its natural preimage in $\K$.

\paragraph{Generator selection: }
Finding a generator of $\F_{p^n}^*$ requires factoring $p^n -1$ which is out of reach. Instead one chooses a random element $g \in \F_{p^n}^*$ and tests if $g^{(p^n - 1)/m} \not = 1$ for all $m$ running over small divisors of $p^n - 1$ (say all divisors smaller than $10^9$). Such an element has a very high probability of being a generator of $\F_{p^n}^*$, and is called a pseudo generator.
Running our code that is available at \cite{Code:smoothness}, we find the following pseudo generator of $\F_{p^{28}}^*$:
$g = 44501\alpha^{27} + 17288\alpha^{26} + 79714\alpha^{25} + 15355\alpha^{24} + 100146\alpha^{23} + 87012\alpha^{22} + 18126\alpha^{21} + 125995\alpha^{20} + 12941\alpha^{19} + 86746\alpha^{18} + 22260\alpha^{17} + 8816\alpha^{16} + 41799\alpha^{15} + 19116\alpha^{14} + 45121\alpha^{13} + 116926\alpha^{12} + 11767\alpha^{11} + 64435\alpha^{10} + 16296\alpha^9 + 33812\alpha^8 + 96819\alpha^7 + 40474\alpha^6 + 105343\alpha^5 + 71563\alpha^4 + 48599\alpha^3 + 102954\alpha^2 + 36712\alpha + 3594$.

\paragraph{Target selection: }
We choose a target constructed from the decimal digits of $\pi$.

$T = 1415926\alpha^{27} + 5358979\alpha^{26} + 3238462\alpha^{25} + 6433832\alpha^{24} + 7950288\alpha^{23} + 4197169\alpha^{22} + 3993751\alpha^{21} + 582097\alpha^{20} + 4944592\alpha^{19} + 3078164\alpha^{18} + 628620\alpha^{17} + 8998628\alpha^{16} + 348253\alpha^{15} + 4211706\alpha^{14} + 7982148\alpha^{13} + 865132\alpha^{12} + 8230664\alpha^{11} + 7093844\alpha^{10} + 6095505\alpha^9 + 8223172\alpha^8 + 5359408\alpha^7 + 1284811\alpha^6 + 1745028\alpha^5 + 4102701\alpha^4 + 9385211\alpha^3 + 555964\alpha^2 + 4622948\alpha + 9549303$.

After reducing each coefficient modulo $p$, the target $T$ becomes:
$T = 104916\alpha^{27} + 114939\alpha^{26} + 92038\alpha^{25} + 9883\alpha^{24} + 84228\alpha^{23} + 1937\alpha^{22} + 60721\alpha^{21} + 57693\alpha^{20} + 93855\alpha^{19} + 62841\alpha^{18} + 104216\alpha^{17} + 83760\alpha^{16} + 86051\alpha^{15} + 16474\alpha^{14} + 116088\alpha^{13} + 78526\alpha^{12} + 102402\alpha^{11} + 14390\alpha^{10} + 64859\alpha^9 + 94910\alpha^8 + 115368\alpha^7 + 104902\alpha^6 + 40715\alpha^5 + 38570\alpha^4 + 77040\alpha^3 + 31560\alpha^2 + 34413\alpha + 110031$. 

\paragraph{Outputs: }

To run our code \cite{Code:smoothness}, one starts by creating an instance from smoothness.sage: diag = Smoothness(p, n, f1, f1, g), and then one calls the method smoothness\_lattice\_n: 
$R_1$, $R_2$, $s_{best}$ = diag.smoothness\_lattice\_n(T). We get $\overline{R_1}$ the output of the algorithm of~\cite{guillevic_descent} (i.e: Algorithm~\ref{our_algorithm} with $s=0$) and $\overline{R_2}$ the output of Algorithm~\ref{our_algorithm} for the best choice of $s$, that is $s_{practical}$. We recall that $s_{practical}$ is the number of columns erased from the lattice  that results in the output of the smallest element, that is $\overline{R_2}$. We get:

$\overline{R_1} = -13x^{27} - 51x^{26} - 10x^{25} + 100x^{24} + 219x^{23} + 80x^{22} + 98x^{21} + 54x^{20} - 5x^{19} + 113x^{18} - 195x^{17} + 92x^{16} - 46x^{15} - 99x^{14} + 9x^{13} + 77x^{12} - 173x^{11} + 77x^{10} + 57x^9 + 213x^8 - 82x^7 - 107x^6 - 76x^5 - 58x^4 - 8x^3 + 34x^2 - 64x - 28$.

$\overline{R_2} = 175x^{23} - 87x^{22} - 10x^{21} + 305x^{20} + 233x^{19} - 37x^{18} - 151x^{17} - 123x^{16} - 30x^{15} + 105x^{14} + 145x^{13} - 214x^{12} + 143x^{11} + 432x^{10} + 63x^9 - 222x^8 - 17x^7 - 303x^6 - 309x^5 - 239x^4 + 25x^3 - 373x^2 - 330x - 174$, where $s_{practical} = 4$.

\paragraph{Norms of the target and the outputs: }
The norm of the target is $\N\left( \bar T \right) \approx 2^{769}$, the norm of $\overline{R_1}$ is $\N\left( \overline{R_1} \right) \approx 2^{507}$, and the norm of $R_2$ is $\N\left( \overline{R_2} \right) \approx 2^{492}$. Our algorithm outputs here an element of norm 15 bits smaller than the one output by \cite{guillevic_descent}.
We emphasize that $\overline{R_1}$ is of degree maximal $27$ whereas $\overline{R_2}$ is of degree $27-4 = 23$ and has slightly larger coefficients.

\paragraph{Probability of smoothness: }Fix a smoothness bound $B = 2^{35}$. Then using the dickman\_rho function implemented is sage, the probability of $\N\left(\overline{R_1}\right)$ being $B$-smooth is about $6.45 \times 10^{-19}$ and the probability of $\N\left(\overline{R_2}\right)$ being $B$-smooth is about $3.77 \times 10^{-18}$. Our output is $5.8$ more likely to be smooth.

\paragraph{Larger example: } As shown in Section~\ref{sec:practical}, our algorithm performs the best as the degree extension $n$ grows. For instance let us look at the $2050$-bits finite field $\F_{p^n} = \F_{2199023255579^{50}}$. All the parameters for this setting, such as the polynomials selected and the generator, can be found in the GitLab repository \cite{Code:smoothness}. Similarly as above, applying Algorithm~\ref{our_algorithm} leads to the following:
\begin{enumerate}
\item The norm of the target chosen with the decimals of $\pi$ is $\N\left(\overline{T}\right) \approx 2^{3152}$
\item The norm of the output $\overline{R_1}$ of \cite{guillevic_descent}'s algorithm is $\N \left( \overline{R_1} \right) \approx 2^{2138}$
\item The norm of the output $\overline{R_2}$ of Algorithm~\ref{our_algorithm} with the best $s$ is $\N \left( \overline{R_2} \right) \approx 2^{2121}$, where the best $s$ is $s_{practical} = 4$. 
\end{enumerate}
In this example our output is $2^{17}$ times smaller. If the smoothness bound is set to $B=2^{70}$,  then our output is about $3.5$ times more likely to be $B$-smooth.
Since the smoothness probability is higher, one can set a lower smoothness bound in order to get a smaller descent tree. 

\section{Data}
\label{Appendix:Data}
The next two tables present the results of our experiments: $n$ is the extension degree, $d$ is the largest divisor of $n$, \textbf{$p$ in bits} is the number of bits of the characteristic $p$, \textbf{Bitsize of the field} is the size of the finite field $\F_{p^n}$ in bits, \textbf{Input norms in bits} is the mean in bits of the norms in the number field of the 1000 targets, \textbf{Output norms with [Guil19] in bits} is the mean in bits of the norms output by [Guil19], \textbf{Our norms in bits} is the mean in bits of the norms output by Algorithm~$1$, $s_{practical}$ is the mean of the best choice 
of~$s$ in Algorithm~$1$ in practice rounded to one decimal place, and $s_{theoretical}$ is the optimal $s$ given from the asymptotic formula rounded to the integer below. 
Each given norm in a given finite field is a mean of the norms of 1000 elements.
The data is sorted in respect to $n$ the extension degree.
Moreover, the polynomials selected for the experiments, the pseudo generators of the multiplicative group in each finite field, and the implementation that produced this data are available at~\cite{Code:smoothness}.

\begin{table}[h!]
\label{tabel_500}
\centering
\begin{tabular}{|c|c|c|c|c|c|c|c|c|}
\hline
\textbf{n} & \textbf{d} & \textbf{p in bits} & \textbf{Bitsize of} & \textbf{Input norms} & \textbf{Output norms} & \textbf{Our norms} & \textbf{$s_{practical}$} & \textbf{$s_{theoretical}$} \\
& & & \textbf{the field}  & \textbf{in bits} & \textbf{with [Gui19] in bits} & \textbf{in bits} & & \\
\hline
4&2&125&500&688&438&438&0.0&0\\
\hline
6&3&83&498&707&457&457&0.1&0\\
\hline
8&4&62&496&721&471&470&0.3&0\\
\hline
9&3&55&495&726&557&557&0.0&0\\
\hline
10&5&50&500&732&479&477&0.4&0\\
\hline
12&6&41&492&734&482&478&0.7&0\\
\hline
14&7&35&490&739&486&481&0.9&0\\
\hline
15&5&33&495&745&573&571&0.2&0\\
\hline
16&8&31&496&755&501&495&1.2&0\\
\hline
18&9&27&486&746&497&490&1.4&0\\
\hline
20&10&25&500&775&517&507&1.7&0\\
\hline
21&7&23&483&751&583&579&0.6&0\\
\hline
22&11&22&484&757&508&498&2.1&0\\
\hline
24&12&20&480&761&511&498&2.5&2\\
\hline
25&5&20&500&787&678&674&0.4&0\\
\hline
26&13&19&494&786&532&520&2.8&3\\
\hline
27&9&18&486&773&600&591&1.1&1\\
\hline
28&14&17&476&769&516&499&3.2&6\\
\hline
30&15&16&480&778&527&509&3.9&8\\
\hline
32&16&15&480&779&525&507&4.2&10\\
\hline
33&11&15&495&815&641&630&1.8&7\\
\hline
34&17&14&476&789&540&517&5.0&12\\
\hline
35&7&14&490&810&704&695&1.1&5\\
\hline
36&18&13&468&779&536&514&5.5&14\\
\hline
38&19&13&494&816&558&535&5.7&15\\
\hline
39&13&12&468&772&606&592&2.7&11\\
\hline
40&20&12&480&801&548&520&6.5&18\\
\hline
42&21&11&462&779&534&507&7.2&19\\
\hline
44&22&11&484&792&542&520&6.9&20\\
\hline
45&15&11&495&840&668&651&4.0&13\\
\hline
46&23&10&460&773&533&506&8.2&21\\
\hline
48&24&10&480&842&594&558&10.4&22\\
\hline
49&7&10&490&836&770&756&2.5&5\\
\hline
50&25&10&500&837&584&554&8.8&23\\
\hline
\end{tabular}
\caption{
Experiments are run on  460 to 500-bit finite fields. The extension degree varies from $4$ to $50$.} 
\end{table}

\begin{table}[h!]
\label{table_2048}
\centering
\begin{tabular}{|c|c|c|c|c|c|c|c|c|}
\hline
\textbf{n} & \textbf{d} & \textbf{p in bits} & \textbf{Bitsize of} & \textbf{Input norms} & \textbf{Output norms} & \textbf{Our norms} & \textbf{$s_{practical}$} & \textbf{$s_{theoretical}$} \\
& & & \textbf{the field}  & \textbf{in bits} & \textbf{with [Gui19] in bits} & \textbf{in bits} & & \\
\hline
4&2&513&2052&2821&1796&1796&0.0&0\\
\hline
6&3&342&2052&2907&1878&1878&0.0&0\\
\hline
8&4&257&2056&2959&1930&1930&0.0&0\\
\hline
9&3&228&2052&2968&2282&2282&0.0&0\\
\hline
10&5&205&2050&2979&1950&1950&0.0&0\\
\hline
12&6&171&2052&3003&1973&1972&0.1&0\\
\hline
14&7&147&2058&3028&1996&1994&0.2&0\\
\hline
15&5&137&2055&3031&2342&2342&0.0&0\\
\hline
16&8&129&2064&3049&2011&2009&0.3&0\\
\hline
18&9&114&2052&3046&2013&2010&0.5&0\\
\hline
20&10&103&2060&3063&2025&2021&0.6&0\\
\hline
21&7&98&2058&3063&2370&2370&0.0&0\\
\hline
22&11&94&2068&3086&2047&2042&0.7&0\\
\hline
24&12&86&2064&3080&2040&2032&0.8&0\\
\hline
25&5&82&2050&3075&2656&2656&0.0&0\\
\hline
26&13&79&2054&3083&2047&2041&0.8&0\\
\hline
27&9&76&2052&3083&2389&2388&0.1&0\\
\hline
28&14&74&2072&3111&2063&2055&1.0&0\\
\hline
30&15&69&2070&3121&2078&2068&1.1&0\\
\hline
32&16&65&2080&3145&2099&2088&1.3&0\\
\hline
33&11&63&2079&3142&2440&2437&0.2&0\\
\hline
34&17&61&2074&3145&2101&2089&1.4&0\\
\hline
35&7&59&2065&3132&2707&2706&0.1&0\\
\hline
36&18&57&2052&3109&2075&2062&1.6&0\\
\hline
38&19&54&2052&3126&2094&2080&1.9&0\\
\hline
39&13&53&2067&3146&2449&2444&0.3&0\\
\hline
40&20&52&2080&3172&2125&2111&1.9&0\\
\hline
42&21&49&2058&3151&2114&2099&2.2&0\\
\hline
44&22&47&2068&3167&2126&2109&2.4&0\\
\hline
45&15&46&2070&3175&2476&2469&0.5&0\\
\hline
46&23&45&2070&3182&2144&2124&2.6&0\\
\hline
48&24&43&2064&3182&2150&2128&2.9&2\\
\hline
49&7&42&2058&3169&2873&2867&0.2&0\\
\hline
50&25&41&2050&3166&2144&2119&3.2&4\\
\hline

\end{tabular}
\caption{
Experiments are run on 2050 to 2080-bit finite fields. The extension degree varies from $4$ to $50$.} 
\end{table}

\end{document}